\documentclass[12pt,a4paper]{amsart}

\usepackage[english]{babel}
\usepackage[utf8x]{inputenc}
\usepackage{amsfonts}
\usepackage{amssymb}
\usepackage{amsmath}
\usepackage{amsrefs}
\usepackage{mathrsfs}
\usepackage{graphicx}
\usepackage{framed}
\usepackage{mdframed}
\usepackage[colorinlistoftodos]{todonotes}
\usepackage{color}

\usepackage{dsfont}
\usepackage{stmaryrd}
\usepackage[scr=boondox]{mathalfa}

\usepackage{float}

\usepackage{tikz-cd}

\newtheorem{theorem}{Theorem}[section]
\newtheorem{lemma}[theorem]{Lemma}
\newtheorem{proposition}[theorem]{Proposition}

\newtheorem{corollary}[theorem]{Corollary}
\newtheorem*{theorem*}{Theorem}

\theoremstyle{definition}
\newtheorem{definition}[theorem]{Definition}

\DeclareMathOperator{\Ad}{Ad}

\DeclareMathOperator{\Iso}{I}
\DeclareMathOperator{\Aff}{Aff}
\DeclareMathOperator{\Orth}{O}

\DeclareMathOperator{\GLin}{GL}

\DeclareMathOperator{\Path}{Path}

\newcommand{\MC}[1]{\omega_{{}_{#1}}}
\newcommand{\Lt}[1]{\operatorname{L}_{#1}}
\newcommand{\Rt}[1]{\operatorname{R}_{#1}}
\newcommand{\quot}[1]{\mathrm{q}_{{}_{#1}}}

\usepackage[all]{xy}

\usepackage[normalem]{ulem}

\DeclareFontEncoding{T3}{}{}
\DeclareFontSubstitution{T3}{cmr}{m}{n}
\DeclareSymbolFont{tipa}{T3}{cmr}{m}{sl}
\SetSymbolFont{tipa}{bold}{T3}{cmr}{bx}{sl}
\DeclareMathSymbol{\kgf}{\mathord}{tipa}{'255}

\newenvironment{smallbmatrix}
  {\left[\begin{smallmatrix}}
  {\end{smallmatrix}\right]}

\usepackage{hyperref}

\title[Completeness for Cartan geometries]{Inequivalence of the various notions of completeness for Cartan geometries}
\author{Jacob W. Erickson \and Benjamin McKay}
\thanks{This material is based upon work done while the first author was under the partial support of the National Science Foundation under Grant No. 2203493.}

\begin{document}
\begin{abstract}We demonstrate that the three standard notions of completeness for Cartan geometries are not actually equivalent, modernizing and extending results due to Yeaton Clifton. In an effort to rectify the folklore on the topic, we also provide detailed explanations of how specious proofs in the literature that would claim that these types of completeness are equivalent break down. In the affine case, we construct geodesically complete homogeneous examples with an inextensible exponential curve for each connected Lie group of dimension at least 2, as well as an infinite-dimensional family of examples whose exponential curves are defined for all time but whose overlying frame bundle is not complete as a metric space with respect to a natural metric for the Cartan connection.\end{abstract}

\maketitle

\section{Introduction}
In Riemannian geometry, completeness is a comfortable and fruitful assumption. All compact Riemannian manifolds are complete, and even among noncompact manifolds, there is no sparsity of examples where the condition holds. Moreover, by the Hopf--Rinow theorem, metric completeness for (torsion-free) Riemannian manifolds happens to be equivalent to geodesic completeness, and since we can convert between the two types of completeness indiscriminately, the assumption gives us plenty of utility for relatively little cost.

While most Cartan geometries, outside the Riemannian case, lack natural metrics on their base manifolds, there are a few ways in which one might hope to recapture the notion of completeness for this wider setting.

First, when a model $(G,H)$ is \emph{reductive}---meaning there exists an $\Ad_H$-invariant subspace $\mathfrak{m}\subseteq\mathfrak{g}$ complementary to $\mathfrak{h}$, so that $\mathfrak{g}=\mathfrak{m}\oplus\mathfrak{h}$ as an $H$-representation---we can define a notion of \emph{geodesic} for Cartan geometries $(\mathscr{G},\omega)$ of type $(G,H)$ by taking exponential curves in $\mathscr{G}$ of the form $t\mapsto\exp(t\omega^{-1}(X))\mathscr{g}$ for $X\in\mathfrak{m}$ and $\mathscr{g}\in\mathscr{G}$, and then projecting down to the base manifold. In the Riemannian case, this definition is equivalent to the usual notion of geodesic, and just like we do for Riemannian manifolds, we can define a reductive Cartan geometry to be \emph{geodesically complete} whenever all of these geodesics exist for all time.

More generally, even outside of the reductive case, we can ask that \textit{all} exponential curves $t\mapsto\exp(t\omega^{-1}(X))\mathscr{g}$, where $\mathscr{g}\in \mathscr{G}$ and $X$ is an arbitrary element of $\mathfrak{g}$, exist for all time; in the modern terminology for Cartan geometries, this condition is simply called \emph{completeness}, though for clarity we will refer to it here as \emph{exponential-completeness}.

Finally, even though there might not be a natural metric on the base manifold, Cartan geometries $(\mathscr{G},\omega)$ of type $(G,H)$ always admit natural Riemannian metrics on their principal $H$-bundle $\mathscr{G}$: for any inner product $\mathrm{g}$ on $\mathfrak{g}$, define $\mathrm{g}_\omega$ to be the Riemannian metric on $\mathscr{G}$ given by $\mathrm{g}_\omega(\xi_1,\xi_2):=\mathrm{g}(\omega(\xi_1),\omega(\xi_2))$. Such metrics are called \emph{$\omega$-constant metrics}, and we can define a Cartan geometry to be \emph{$\omega$-constant metric complete}, or \emph{development-complete} (see Theorem \ref{devcompismetcomp}), whenever one of these metrics is complete; in the Riemannian setting, this is equivalent to metric completeness of the base manifold.

Lamentably, while all three of these notions of completeness are valid and meaningful for more general Cartan geometries, most hopes for using them outside of the Riemannian setting are in vain. In a majority of the cases of interest, every relevant notion of completeness seems to be far too restrictive to be applicable outside of a limited number of examples. In the holomorphic category, for example, the results in \cite{McKayComplexComp} tell us that an exponential-complete or development-complete Cartan geometry of parabolic type is necessarily isomorphic to its own model geometry.

Even with this limitation on its applicability, one might still hope that a version of the Hopf--Rinow theorem holds, so that these various notions of completeness happen to coincide in general, just like they do in the Riemannian case. Indeed, this supposed equivalence seems widely accepted in the subject's folklore, and both of the authors of this paper have previously written preprints that included attempts to prove this result. Both such preprints independently followed the same faulty reasoning used by Kobayashi when he introduced this equivalence result via Theorems 5.6 and 5.15 of \cite{Kobayashi1957}, and in an effort to prevent spreading this erroneous result further, we have written this paper to help correct the folklore: in general, none of these three notions of completeness are equivalent.

This inequivalence was essentially already proven by Clifton in \cite{Clifton1966}, though a combination of unclear explanations and outdated machinery and terminology seems to have prevented that paper from receiving the recognition it deserves. We hope that, by translating the content of that paper for a more modern audience, we can also help to rectify this disservice.

Our main results provide recipes for constructing a cornucopia of Cartan geometries that satisfy some completeness assumption without also satisfying a stronger one. For example, we have the following results for affine geometries.

\begin{theorem*}[Corollary \ref{cliftongen} below]
On each connected Lie group $S$ of dimension at least 2, there exists an $S$-invariant affine geometry on $S$ that is geodesically complete but not exponential-complete.\end{theorem*}

\begin{theorem*}[Proposition \ref{loopexamples1} and Corollary \ref{loopexamples2} below]
Consider the Klein geometry $(\Aff(m),\MC{\Aff(m)})$ for affine geometry on $\mathbb{R}^m$. If $\beta$ is a (smooth) $1$-form on $\mathbb{R}^m$ with compact support, then the affine geometry $(\Aff(m),\MC{\Aff(m)}+\quot{\GLin_m\mathbb{R}}^*\beta\otimes\mathds{1})$ over $\mathbb{R}^m$ is exponential-complete, but it is development-complete if and only if $\mathrm{d}\beta=0$.
\end{theorem*}

Let us briefly outline the structure of the paper, before we proceed. We will start by reviewing some existing results and terminology in Section \ref{prelims}. Section \ref{pathapprox} discusses some important subtleties to using the notion of development to approximate solutions to certain ODEs, and in particular, describes how proofs intended to show an equivalence between the three notions of completeness mentioned above fail. In Section \ref{clifton}, we will review and expound upon the two examples due to Clifton in \cite{Clifton1966}, from which we will derive inspiration for our main results in Section \ref{halfsteady}, which explains how to extend the idea of the Clifton plane to other reductive Cartan geometries, and in Section \ref{rescalingdeformations}, which demonstrates how to construct deformations of affine geometries that are exponential-complete but not development-complete.

\vspace{1em}
\section{Preliminaries}\label{prelims}
In this section, we establish some terminology, notation, and existing results that will be useful in the rest of the article. Note, however, that this would probably work poorly as a first introduction to the topic; for that, we currently recommend either the standard references \cite{Sharpe1997} and \cite{CapSlovakPG1}, or the in-progress work \cite{McKayIntro}.

\subsection{Basic definitions}
A Cartan geometry is a geometric structure that is modeled on a given homogeneous one.

\begin{definition}
A \emph{model geometry} (or \emph{model}) is a pair $(G,H)$, where $G$ is a Lie group and $H\leq G$ is a closed subgroup such that $G/H$ is connected. In a model $(G,H)$, we call $G$ the model group and $H$ the stabilizer or isotropy subgroup.
\end{definition}

\begin{definition}
A \emph{Cartan geometry of type $(G,H)$} over a (smooth)\footnote{Throughout, we will always assume our manifolds to be smooth. We will also generally assume that the base manifold $M$ of a Cartan geometry is connected.} manifold $M$ is a pair $(\mathscr{G},\omega)$, where $\mathscr{G}$ is a principal $H$-bundle over $M$, with quotient map always denoted by $\quot{H}:\mathscr{G}\to M$ for convenience, and $\omega$ is a (smooth) $\mathfrak{g}$-valued 1-form, called a \emph{Cartan connection}, on $\mathscr{G}$ satisfying the following three conditions:
\begin{itemize}
\item For each $\mathscr{g}\in\mathscr{G}$, $\omega_\mathscr{g}:T_\mathscr{g}\mathscr{G}\to\mathfrak{g}$ is a linear isomorphism.
\item For each $h\in H$, $\Rt{h}^*\omega=\Ad_{h^{-1}}\omega$, where $\Rt{h}:\mathscr{g}\mapsto\mathscr{g}h$ denotes right-translation by $h$.
\item For each $Y\in\mathfrak{h}$, the time-$t$ flow of the vector field $\omega^{-1}(Y)$ is given by $\exp(t\omega^{-1}(Y))=\Rt{\exp(tY)}$ for all $t\in\mathbb{R}$.
\end{itemize}
\end{definition}

Riemannian geometry, which we can think of as being modeled on Euclidean geometry $(\Iso(m),\Orth(m))$, where $\Iso(m)\simeq\mathbb{R}^m\rtimes\Orth(m)$ is the group of Euclidean isometries of $\mathbb{R}^m\cong\Iso(m)/\Orth(m)$, gives an archetypal example for how this works. Similarly, we can model manifolds with a choice of affine connection $(M,\nabla)$ on affine geometry $(\Aff(m),\GLin_m\mathbb{R})$, where $\Aff(m)\simeq\mathbb{R}^m\rtimes\GLin_m\mathbb{R}$ is the group of affine transformations of $\mathbb{R}^m\cong\Aff(m)/\GLin_m\mathbb{R}$.

The geometric structure of a model geometry $(G,H)$ can be encoded as a Cartan geometry called the Klein geometry.

\begin{definition}
The \emph{Klein geometry of type $(G,H)$} is defined to be the Cartan geometry $(G,\MC{G})$ of type $(G,H)$ over $G/H$, where $\MC{G}$ is the Maurer--Cartan form $X_g\mapsto\Lt{g^{-1}*}X_g\in T_{g^{-1}g}G=T_eG=\mathfrak{g}$.
\end{definition}

In essence, a Cartan connection $\omega$ compares the first-order behavior of motion in $\mathscr{G}$ with that of the model Lie group $G$. The notion of development allows us to integrate this comparison over paths and curves.

\begin{definition}
In a Cartan geometry $(\mathscr{G},\omega)$ of type $(G,H)$, consider an absolutely continuous curve or path $\gamma:I\to\mathscr{G}$, where $I$ is an interval containing $0$. We define its \emph{development} $\gamma_G:I\to G$ to be the unique curve or path in $G$ satisfying $\gamma_G(0)=e$ and $\omega(\dot{\gamma})=\MC{G}(\dot{\gamma}_G)$.
\end{definition}

We will elaborate more on developments in the next subsection. For piecewise smooth paths, developments are explored in \cite{Sharpe1997}.

Cartan geometries come equipped with a natural notion of geometric curve called an exponential curve, given by flows of $\omega$-constant vector fields $\omega^{-1}(X)$ for some $X\in\mathfrak{g}$.

\begin{definition}
A curve $\gamma:I\to\mathscr{G}$ in a Cartan geometry $(\mathscr{G},\omega)$ of type $(G,H)$ is called an \emph{exponential curve} when its development $\gamma_G$ is of the form $t\mapsto\exp(tX)$ for some $X\in\mathfrak{g}$. Equivalently, $\gamma$ is an exponential curve if and only if it is given by the flow through $\gamma(0)$ of an $\omega$-constant vector field $\omega^{-1}(X)$ for some $X\in\mathfrak{g}$, so that $\gamma(t)=\exp(t\omega^{-1}(X))\gamma(0)$.
\end{definition}

These tend to be quite natural: in the affine case, for example, exponential curves of the form $t\mapsto\exp(t\omega^{-1}(v))\mathscr{g}$, for $\mathscr{g}\in\mathscr{G}$ and $v\in\mathbb{R}^m\trianglelefteq\mathfrak{aff}(m)$ in the ideal corresponding to translations, correspond to geodesics in the base manifold. For reductive Cartan geometries, we have a natural extension of this example.

\begin{definition}A model geometry $(G,H)$ is called \emph{reductive} if and only if there exists an $\Ad_H$-invariant subspace $\mathfrak{m}$ complementary to $\mathfrak{h}$. When specifying that a model is reductive, we will always assume a fixed choice of $\mathfrak{m}$.\end{definition}

\begin{definition}
In a Cartan geometry $(\mathscr{G},\omega)$ of reductive type $(G,H)$, a \emph{geodesic} is an exponential curve of the form $t\mapsto\exp(t\omega^{-1}(X))\mathscr{g}$ for some $X\in\mathfrak{m}$.
\end{definition}

Here and throughout, we will follow the practice of not distinguishing between terminology for curves in the base manifold and those in the overlying principal bundle. This is rather natural and convenient when working with Cartan geometries, and leads to much cleaner wordings for definitions and results. In other words, we will also use the term \emph{geodesic} to refer to a curve of the form $t\mapsto \quot{H}(\exp(t\omega^{-1}(X))\mathscr{g})$, for some $\mathscr{g}\in\mathscr{G}$ and $X\in\mathfrak{m}$, on the base manifold $M$ of a reductive Cartan geometry $(\mathscr{G},\omega)$.

\subsection{Development and continuity}
Recall that a continuous path $\gamma:[0,1]\to M$ is \emph{absolutely continuous} if and only if it is absolutely continuous in each chart, so that $\dot{\gamma}$ is defined almost everywhere on $[0,1]$ and, if $\phi:U\to\mathbb{R}^m$ is a chart and $\gamma([a,b])\subseteq U$, then $\|\phi^*\MC{\mathbb{R}^m}(\dot{\gamma})\|$ is Lebesgue integrable on $[a,b]$ and $\phi(\gamma(s))=\phi(\gamma(a))+\int_a^s\phi^*\MC{\mathbb{R}^m}(\dot{\gamma}(t))\mathrm{d}t$ for all $s\in[a,b]$. In particular, absolutely continuous paths are uniquely determined by their almost-everywhere-defined derivatives. 

For paths $\gamma:[0,1]\to\mathscr{G}$ in a Cartan geometry $(\mathscr{G},\omega)$ of type $(G,H)$, $\gamma(t)\in\mathscr{G}$ and $\omega(\dot{\gamma}(t))\in\mathfrak{g}$ uniquely determine $\dot{\gamma}(t)=\omega_{\gamma(t)}^{-1}\big(\omega(\dot{\gamma}(t))\big)$, so absolutely continuous paths $\gamma$ in $\mathscr{G}$ are uniquely determined by their Darboux derivatives $\omega(\dot{\gamma})\in L^1([0,1];\mathfrak{g})$.

\begin{definition}Picking an inner product $\mathrm{g}$ on the Lie algebra $\mathfrak{g}$, we may equip $L^1([0,1];\mathfrak{g})$ with its usual norm $\|f\|_{L^1}:=\int_0^1\|f(t)\|_\mathrm{g}\mathrm{d}t$ and subsequent metric topology. For $\mathscr{e}\in\mathscr{G}$, we define
\[\Path_\mathscr{e}(\mathscr{G}):=\left\{\gamma\in C^0([0,1];\mathscr{G}):\begin{array}{c}\gamma\text{ is an absolutely continuous} \\ \text{path such that }\gamma(0)=\mathscr{e}\end{array}\right\}\]
to be the \emph{space of (absolutely continuous) paths starting at $\mathscr{e}$}, equipped with the subspace topology from the injection $\Path_\mathscr{e}(\mathscr{G})\to L^1([0,1];\mathfrak{g})$ given by $\gamma\mapsto\omega(\dot{\gamma})$.\end{definition}

Picking an inner product $\mathrm{g}$ on the Lie algebra $\mathfrak{g}$ also allows us to define an \emph{$\omega$-constant metric} $\mathrm{g}_\omega$ on $\mathscr{G}$ by $\mathrm{g}_\omega(\xi_1,\xi_2):=\mathrm{g}(\omega(\xi_1),\omega(\xi_2))$.
This imbues $\mathscr{G}$---though generally not its base manifold $M$---with a natural Riemannian metric structure, hence a natural metric \[\mathrm{dist}_{\mathrm{g}_\omega}(\mathscr{g}_1,\mathscr{g}_2):=\!\!\inf_{\substack{\gamma\in\Path_{\mathscr{g}_1}(\mathscr{G})\\ \gamma(1)=\mathscr{g}_2}}\|\omega(\dot{\gamma})\|_{L^1}\]
on each connected component of $\mathscr{G}$.

By essentially the same argument as Lemma 3.2 from \cite{EnsnaringPaper}, we have the following lemma in this setting.

\begin{lemma}If $\gamma\in\Path_\mathscr{e}(\mathscr{G})$ and $f\in L^1([0,1];\mathfrak{g})$ with $\|f-\omega(\dot{\gamma})\|_{L^1}$ sufficiently small, then there exist a constant $c_\gamma>0$ and $\gamma_f\in\Path_\mathscr{e}(\mathscr{G})$ such that $\omega(\dot{\gamma}_f)=f$ and $\mathrm{dist}_{\mathrm{g}_\omega}(\gamma_f(t),\gamma(t))\leq c_\gamma\|f-\omega(\dot{\gamma})\|_{L^1}$.\end{lemma}

This tells us that the image of $\Path_\mathscr{e}(\mathscr{G})$ in $L^1([0,1];\mathfrak{g})$ is open, and that sequences in $\Path_\mathscr{e}(\mathscr{G})$ that converge with respect to the subspace topology from $L^1([0,1];\mathfrak{g})$ also converge uniformly, making the topology on $\Path_\mathscr{e}(\mathscr{G})$ a natural choice for a space of paths.

\begin{proposition}\label{kleinpathhomeo} For the Klein geometry $(G,\MC{G})$ of type $(G,H)$, the map $\Path_e(G)\to L^1([0,1];\mathfrak{g})$ given by $\gamma\mapsto\MC{G}(\dot{\gamma})$ is a homeomorphism.\end{proposition}
\begin{proof}From the definition of the topology on $\Path_e(G)$, we just need to show that, for every $f\in L^1([0,1];\mathfrak{g})$, there is a corresponding path $\gamma\in\Path_e(G)$ such that $\omega(\dot{\gamma})=f$. To do this, we notice that the vector field $(t,g)\mapsto (\MC{G})^{-1}_g(f(t))$ on $[0,1]\times G$ is measurable in $t$ and smooth in $g$, and on each closed coordinate ball $B$ in $G$, there is some $c_B>0$ such that, for almost every $t\in[0,1]$, $\MC{G}^{-1}(f(t))$ and its first derivatives in chart coordinates are bounded by $c_B\|f(t)\|_\mathrm{g}$, which is integrable in $t$. By Carath\'eodory's theorem from ODEs (see, for example, I.5 of \cite{Hale1969}), there is a unique local flow defined near each $(t,g)\in[0,1]\times G$. We can therefore define the desired $\gamma$ on some interval $[0,\tau)$ with $\tau\leq 1$, and since $\gamma|_{[0,\tau)}$ has $\mathrm{g}_{\MC{G}}$-length bounded by $\|f\|_{L^1}$, the fact that $G$ is complete with respect to the left-invariant Riemannian metric $\mathrm{g}_{\MC{G}}$ means we can continuously extend $\gamma$ to $\tau$, at which point the local flow allows us to extend it further. As such, we get a well-defined $\gamma\in\Path_e(G)$ with $\MC{G}(\dot{\gamma})=f$ for each $f\in L^1([0,1];\mathfrak{g})$.\mbox{\qedhere}\end{proof}

This also lets us recover, and even strengthen, the continuity of the development map noted in Proposition 3.3 of \cite{EnsnaringPaper}.

\begin{corollary}\label{devcont}For every path $\gamma\in\Path_\mathscr{e}(\mathscr{G})$, there exists a unique development $\gamma_{{}_G}\in\Path_e(G)$, and the development map
\[(\cdot)_G:\Path_\mathscr{e}(\mathscr{G})\to\Path_e(G)\]
given by $\gamma\mapsto\gamma_G$ is a continuous embedding.\end{corollary}
\begin{proof}The development map $(\cdot)_G:\Path_\mathscr{e}(\mathscr{G})\to\Path_e(G)$ is precisely the composition of the inverse of the homeomorphism from $\Path_e(G)$ to $L^1([0,1];\mathfrak{g})$ from Proposition \ref{kleinpathhomeo} with the Darboux derivative map from $\Path_\mathscr{e}(\mathscr{G})$ to $L^1([0,1];\mathfrak{g})$ given by $\gamma\mapsto\omega(\dot{\gamma})$, which is an embedding by our choice of topology on $\Path_\mathscr{e}(\mathscr{G})$.\mbox{\qedhere}\end{proof}

If, in particular, we have a sequence $(\gamma_k)$ in $\Path_\mathscr{e}(\mathscr{G})$ converging to some $\gamma_\infty\in\Path_\mathscr{e}(\mathscr{G})$, then the developments $((\gamma_k)_G)$ converge to $(\gamma_\infty)_G$ in $\Path_e(G)$. In Section \ref{pathapprox}, we will investigate when this implication is reversible, so that convergence of paths in the model geometry can tell us something about convergence of their antidevelopments.

\begin{definition}
Given a path $\gamma\in\Path_e(G)$, a path $\gamma_\mathscr{G}\in\Path_\mathscr{e}(\mathscr{G})$ is called the \emph{antidevelopment of $\gamma$ from $\mathscr{e}\in\mathscr{G}$} when $(\gamma_\mathscr{G})_G=\gamma$. In other words, for $\gamma\in\Path_e(G)$ in the image of $(\cdot)_G$, its antidevelopment from $\mathscr{e}$ is the path $(\cdot)_G^{-1}(\gamma)\in\Path_\mathscr{e}(\mathscr{G})$.
\end{definition}

Note that, while developments from a Cartan geometry $(\mathscr{G},\omega)$ into the model Lie group $G$ always exist, antidevelopments from $G$ to $(\mathscr{G},\omega)$ are not guaranteed to exist in general.
Consider, for example, the Cartan geometry $(\quot{\Orth(2)}^{-1}(D),\MC{\Iso(2)})$ of type $(\Iso(2),\Orth(2))$ over the unit disk $D$ embedded in the Euclidean plane: whenever we have a path in $\Iso(2)$ whose projection to $\Iso(2)/\Orth(2)\cong\mathbb{R}^2$ leaves the unit disk centered at the origin, the antidevelopment cannot be well-defined on all of $[0,1]$ because it would need to run off of $D$. This handily brings us to the notions of completeness in the following subsection.

\subsection{Three notions of completeness}
We will say that a geometry is development-complete when antidevelopments are guaranteed to exist.

\begin{definition}
A Cartan geometry $(\mathscr{G},\omega)$ of type $(G,H)$ is called \emph{development-complete} or \emph{dev-complete} if and only if, for some $\mathscr{e}\in\mathscr{G}$, every $\gamma\in\Path_e(G)$ has a well-defined antidevelopment $\gamma_\mathscr{G}\in\Path_\mathscr{e}(\mathscr{G})$.
\end{definition}

By Corollary \ref{devcont} above, this is equivalent to the development map $(\cdot)_G:\Path_\mathscr{e}(\mathscr{G})\to\Path_e(G)$ being a homeomorphism, or equivalently, the Darboux derivative map from $\Path_\mathscr{e}(\mathscr{G})$ to $L^1([0,1];\mathfrak{g})$ being a homeomorphism.

Assuming---as we always will here---that our base manifold $M$ is connected, development-completeness is independent of the choice of starting element $\mathscr{e}\in\mathscr{G}$ for the antidevelopments in $\mathscr{G}$. Indeed, when $(\mathscr{G},\omega)$ is dev-complete for $\mathscr{e}\in\mathscr{G}$, we get an antidevelopment map $(\cdot)_\mathscr{G}=(\cdot)_G^{-1}:\Path_e(G)\to\Path_\mathscr{e}(\mathscr{G})$, and for $h\in H$, each $\gamma\in\Path_e(G)$ has an antidevelopment $(h\gamma h^{-1})_\mathscr{G}h\in\Path_{\mathscr{e}h}(\mathscr{G})$ from $\mathscr{e}h$, since \[((h\gamma h^{-1})_\mathscr{G}h)_G=h^{-1}((h\gamma h^{-1})_\mathscr{G})_Gh=h^{-1}h\gamma h^{-1}h=\gamma.\] Similarly, given a path $\delta:[0,1]\to\mathscr{G}$ from $\delta(0)=\mathscr{e}'$ to $\delta(1)=\mathscr{e}$ for $\mathscr{e},\mathscr{e}'\in\mathscr{G}$, $\gamma\in\Path_e(G)$ has an antidevelopment $\gamma_\mathscr{G}\in\Path_\mathscr{e}(\mathscr{G})$ if and only if the concatenation $\delta_G\star(\delta_G(1)\gamma)\in\Path_e(G)$ has antidevelopment $\delta\star\gamma_\mathscr{G}\in\Path_{\mathscr{e}'}(\mathscr{G})$, where the concatenation $c_1\star c_2:[0,1]\to\mathscr{G}$ of two paths $c_1,c_2:[0,1]\to\mathscr{G}$ such that $c_1(1)=c_2(0)$ is defined by \[(c_1\star c_2)(t):=\left\{\begin{array}{ll} c_1(2t) & \text{if }t\in[0,\tfrac{1}{2}], \\ c_2(2t-1) & \text{if }t\in[\tfrac{1}{2},1].\end{array}\right.\]

By restricting the types of paths for which we require well-defined antidevelopments, we get weaker notions of completeness. For example, a geometry is called \emph{(exponential-)complete} when all of its exponential curves are well-defined on all of $\mathbb{R}$.

\begin{definition}
A Cartan geometry $(\mathscr{G},\omega)$ of type $(G,H)$ is called \emph{exponential-complete} or \emph{exp-complete} if and only if all of its $\omega$-constant vector fields are complete, meaning that for each $X\in\mathfrak{g}$, the flow given by $t\mapsto\exp(t\omega^{-1}(X))$ is well-defined for all $t\in\mathbb{R}$.
\end{definition}

If $(\mathscr{G},\omega)$ is dev-complete, then it must also be exp-complete, because for each $X\in\mathfrak{g}$, the path $t\mapsto\exp(t\omega^{-1}(X))\mathscr{g}$, if well-defined, is the antidevelopment of $t\mapsto\exp(tX)$ from $\mathscr{g}\in\mathscr{G}$.

When $(G,H)$ is reductive, we can further restrict to just requiring all geodesics be well-defined on all of $\mathbb{R}$.

\begin{definition}
A Cartan geometry $(\mathscr{G},\omega)$ of reductive type $(G,H)$ is \emph{geodesically complete} if and only if all of the $\omega$-constant vector fields $\omega^{-1}(X)$ for which $X\in\mathfrak{m}$ are complete, or equivalently, if and only if every geodesic is defined for all time.
\end{definition}

As with development-completeness and exponential-completeness, an exp-complete Cartan geometry $(\mathscr{G},\omega)$ of reductive type $(G,H)$ is necessarily geodesically complete because all geodesics are exponential curves. In short, dev-complete implies exp-complete, and exp-complete implies geodesically complete.

We have one more notion of completeness, but as we will see, it is precisely equivalent to development-completeness.

\begin{definition}
A Cartan geometry $(\mathscr{G},\omega)$ of type $(G,H)$ is said to be \emph{$\omega$-constant metric complete} if and only if the metric space on $\mathscr{G}$ induced by an $\omega$-constant metric is complete.
\end{definition}

\begin{theorem}[Theorem 1 of \cite{Clifton1966}, rephrased]\label{devcompismetcomp}
For a Cartan geometry $(\mathscr{G},\omega)$ of type $(G,H)$, $\omega$-constant metric completeness is equivalent to development-completeness.
\end{theorem}
\begin{proof}
Suppose $(\mathscr{G},\omega)$ is $\omega$-constant metric complete and $\gamma\in\Path_e(G)$. For every $\tau\in(0,1]$ for which we have a well-defined antidevelopment $\gamma_\mathscr{G}|_{[0,\tau)}:[0,\tau)\to\mathscr{G}$ starting at $\gamma_\mathscr{G}(0)=\mathscr{e}$, $\gamma_\mathscr{G}|_{[0,\tau)}$ has finite $\mathrm{g}_\omega$-length, so it has a unique limit $\lim_{t\rightarrow\tau}\gamma_\mathscr{G}|_{[0,\tau)}(t)$, so $\gamma_\mathscr{G}|_{[0,\tau)}$ can be continuously extended to $[0,\tau]$. As such, we get a well-defined antidevelopment $\gamma_\mathscr{G}\in\Path_\mathscr{e}(\mathscr{G})$ for $\gamma$, hence $(\mathscr{G},\omega)$ is development-complete.

Conversely, suppose $(\mathscr{G},\omega)$ is not $\omega$-constant metric complete. Then, by the Hopf--Rinow theorem, we know that there exists an inextensible constant-speed geodesic $\gamma:[0,1)\to\mathscr{G}$ starting at some $\gamma(0)=\mathscr{e}\in\mathscr{G}$. Because the development $\gamma_G:[0,1)\to G$ has finite $\mathrm{g}_{\MC{G}}\!\!$-length---equal to the constant speed $\|\omega(\dot{\gamma})\|_\mathrm{g}$ of $\gamma$ in $\mathscr{G}$---and the Klein geometry $(G,\MC{G})$ is $\MC{G}$-constant metric complete, $\gamma_G$ continuously extends to a path in $\Path_e(G)$ that has no antidevelopment in $\Path_\mathscr{e}(\mathscr{G})$, so $(\mathscr{G},\omega)$ cannot be development-complete.\mbox{\qedhere}
\end{proof}

It is worth noting that, in the above result, we needed our paths to be absolutely continuous rather than just, for example, piecewise smooth. Indeed, while the development $\gamma_G$ of the constant speed geodesic in the second part of the proof admits a continuous extension to $[0,1]$, so that $\gamma_G$ is absolutely continuous, there is no guarantee that the end result has any smoothness at the endpoint.

\vspace{1em}
\section{Path approximation via development}\label{pathapprox}
As we mentioned above, dev-complete Cartan geometries are always exp-complete, and for reductive Cartan geometries, exp-completeness always implies geodesic completeness. However, none of these three notions of completeness are equivalent in general. Faulty proofs that would seem to show otherwise all hinge, in one way or another, on the same alluring and persuasive idea: approximation of paths using development.

Crucially, this idea of approximating paths through development is only \textit{slightly} wrong, and we believe that the core of it is worth salvaging: with a solid understanding of its limitations, the idea can be made to work in many situations.

\subsection{Intuition and limitations}
Suppose we have an ODE on $\mathbb{R}^m$ of the form $\dot{x}(t)=f(t)$, for some continuous function $f:[0,1]\to\mathbb{R}^m$, and we want to determine whether it has a well-defined solution with the initial condition $x(0)=x_0$. Rather than finding an exact solution, which might be quite difficult, we can instead apply the standard Euler approximation scheme: starting at $x(0)=x_0$, we move for some small amount of time $\delta>0$ along the line through $x_0$ with velocity $f(0)$, and then, once we get to $x_0+\delta f(0)$, we change directions and move for time $\delta$ along the line through $x_0+\delta f(0)$ with velocity $f(\delta)$, until we reach $x_0+\delta f(0)+\delta f(\delta)$, and so on. The result of this iterative procedure, pictured in Figure \ref{euler}, is a concatenation of line segments approximating the solution of our ODE through $x_0$. By picking $\delta$ sufficiently small, we can make this approximation arbitrarily close to our desired solution, and in particular, we can prove that the solution to the ODE starting at $x_0$ is well-defined on $[0,1]$ by showing that these approximations by line segments converge.

\begin{figure}[h]
\centering\includegraphics[width=0.6\textwidth]{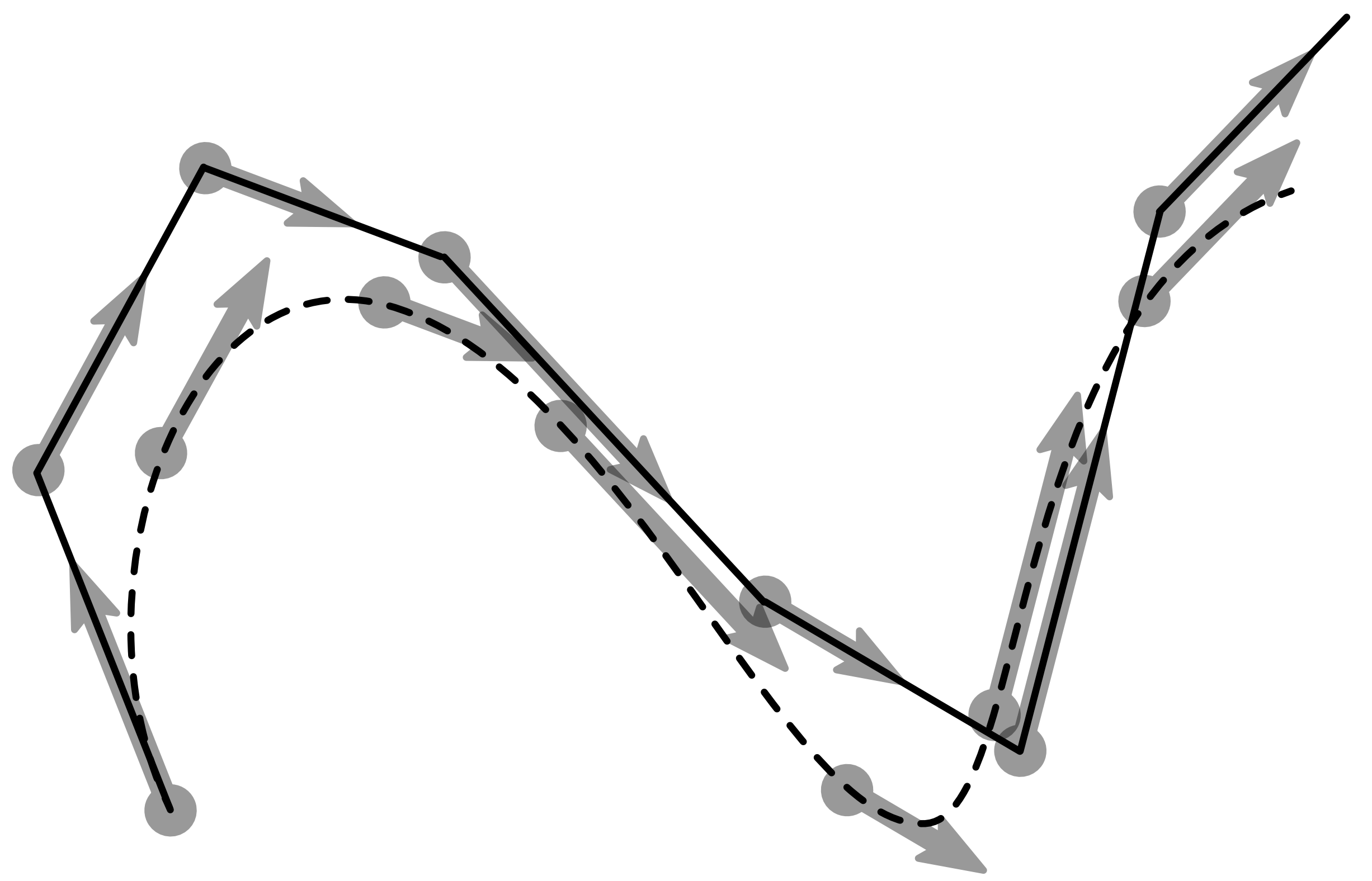}
\caption{An illustration of the Euler approximation method, with the dotted path (representing the solution of an ODE) being approximated by a concatenation of line segments}
\label{euler}
\end{figure}

Let us write this more suggestively: fixing $\delta>0$, we can recursively define $x_{k+1}=\exp(\delta\MC{\mathbb{R}^m}^{-1}(f(k\delta)))x_k$, the time-$\delta$ flow of the vector field $\MC{\mathbb{R}^m}^{-1}(f(k\delta))$ starting at $x_k$, so that the path \[\gamma_\delta:t\mapsto\exp\left((t-\delta\lfloor t/\delta\rfloor)\MC{\mathbb{R}^m}^{-1}(f(\delta\lfloor t/\delta\rfloor))\right)x_{\lfloor t/\delta\rfloor}\] gives the corresponding Euler approximation. In other words, Euler approximation works by approximating the solution of an ODE by a concatenation of exponential paths for the Klein geometry $(\mathbb{R}^m,\MC{\mathbb{R}^m})$ of type $(\mathbb{R}^m,\{0\})$.

Of course, this leads us to ask whether a similar procedure works with arbitrary Lie groups $G$, and hence to the following result.

\begin{proposition}\label{geuler}
Suppose that $f\in L^1([0,1];\mathfrak{g})$, and for each $\delta\in(0,1]$ and nonnegative integer $k<1/\delta$, set
\[f_{k,\delta}:=\left\{\begin{array}{ll}\tfrac{1}{\delta}\int_{k\delta}^{(k+1)\delta}f(t)\,\mathrm{d}t & \text{if }k\leq\tfrac{1}{\delta}-1, \\ \tfrac{1}{1-\delta\lfloor1/\delta\rfloor}\int_{\delta\lfloor1/\delta\rfloor}^{1}f(t)\,\mathrm{d}t & \text{if }k=\lfloor1/\delta\rfloor.\end{array}\right.\]
If, for each fixed $\delta$, we recursively set $g_0:=e$ and $g_{k+1}:=g_k\exp(\delta f_{k,\delta})$ to define $\gamma_\delta\in\Path_e(G)$ by
\begin{align*}\gamma_\delta(t) & :=g_{\lfloor t/\delta\rfloor}\exp\!\Big((t-\delta\lfloor t/\delta\rfloor)f_{\lfloor t/\delta\rfloor,\delta}\Big) \\ & =\exp\!\Big((t-\delta\lfloor t/\delta\rfloor)\,\MC{G}^{-1}(f_{\lfloor t/\delta\rfloor,\delta})\Big)g_{\lfloor t/\delta\rfloor},\end{align*}
then as $\delta\rightarrow 0$, the paths $\gamma_\delta$ converge to a solution $\gamma\in\Path_e(G)$ for the ODE $\MC{G}(\dot{\gamma})=f$. For $f$ continuous, we may replace $f_{k,\delta}$ by $f(k\delta)$.
\end{proposition}
\begin{proof}
By construction, $\MC{G}(\dot{\gamma}_\delta(t))=f_{\lfloor t/\delta\rfloor,\delta}$. For $k\leq\tfrac{1}{\delta}-1$, $f_{k,\delta}$ is the difference quotient
\[f_{k,\delta}=\frac{1}{\delta}\int_{k\delta}^{(k+1)\delta}f(t)\,\mathrm{d}t=\frac{1}{\delta}\left(\int_0^{(k+1)\delta}f(t)\,\mathrm{d}t-\int_0^{k\delta}f(t)\,\mathrm{d}t\right),\]
and likewise $f_{\lfloor1/\delta\rfloor,\delta}$ is the difference quotient over the remaining part of the interval, so since $f(t)=\partial_t\int_0^t f(\tau)\,\mathrm{d}\tau$ almost everywhere, $f_{\lfloor t/\delta\rfloor,\delta}$ converges to $f(t)$ for almost every $t\in[0,1]$ as $\delta\rightarrow 0$. In particular, $\|f_{\lfloor t/\delta\rfloor,\delta}\|_\mathrm{g}+\|f(t)\|_\mathrm{g}-\|f(t)-f_{\lfloor t/\delta\rfloor,\delta}\|_\mathrm{g}$ converges to $2\|f(t)\|_\mathrm{g}$ for almost every $t\in[0,1]$, so by Fatou's lemma,
\begin{align*}\int_0^1 2\|f(t)\|_\mathrm{g}\,\mathrm{d}t
& \leq\liminf_{\delta\rightarrow 0}\int_0^1(\|f_{\lfloor t/\delta\rfloor,\delta}\|_\mathrm{g}+\|f(t)\|_\mathrm{g}-\|f(t)-f_{\lfloor t/\delta\rfloor,\delta}\|_\mathrm{g})\mathrm{d}t \\
& \leq\int_0^1 2\|f(t)\|_\mathrm{g}\,\mathrm{d}t-\limsup_{\delta\rightarrow 0}\|f-\MC{G}(\dot{\gamma}_\delta)\|_{L^1},\end{align*}
hence $\limsup_{\delta\rightarrow 0}\|f-\MC{G}(\dot{\gamma}_\delta)\|_{L^1}=0$. By Proposition \ref{kleinpathhomeo}, it follows that $\gamma_\delta$ converges to a $\gamma\in\Path_e(G)$ with $\MC{G}(\dot{\gamma})=f$.

When $f:[0,1]\to\mathfrak{g}$ is continuous, we can use $f(k\delta)$ instead of $f_{k,\delta}$, since $f(\delta\lfloor t/\delta\rfloor)\rightarrow f(t)$ pointwise and consequently $\|f-\MC{G}(\dot{\gamma}_\delta)\|_{L^1}\rightarrow 0$ by dominated convergence.\mbox{\qedhere}\end{proof}

In other words, whenever we have a path $\gamma\in\Path_e(G)$, we can always approximate it by a concatenation $\gamma_\delta\in\Path_e(G)$ of exponential paths. Consequently, if we wanted to show that $\gamma$ has a well-defined antidevelopment $\gamma_\mathscr{G}\in\Path_\mathscr{e}(\mathscr{G})$ in a Cartan geometry $(\mathscr{G},\omega)$ of type $(G,H)$, then one way to do it would be to use this approximation by exponential paths to find it: if the antidevelopments $(\gamma_\delta)_\mathscr{G}\in\Path_\mathscr{e}(\mathscr{G})$ are well-defined and converge in $\Path_\mathscr{e}(\mathscr{G})$, then their limit will be $\gamma_\mathscr{G}$ by Corollary \ref{devcont}.

Moreover, whenever $(\mathscr{G},\omega)$ is exp-complete, we already know that the antidevelopments $(\gamma_\delta)_\mathscr{G}\in\Path_\mathscr{e}(\mathscr{G})$ must be well-defined because they are concatenations of exponential paths. This is the (valid) key idea behind the various (faulty) proofs of the supposed equivalence between exp-completeness and dev-completeness, which we distill into the following corollary of Proposition \ref{geuler}.

\begin{corollary}
Suppose $(\mathscr{G},\omega)$ is a Cartan geometry of type $(G,H)$ and $\gamma\in\Path_e(G)$. If $(\mathscr{G},\omega)$ is exp-complete, then for each $\delta>0$, the concatenated path $\gamma_\delta\in\Path_e(G)$ defined in Proposition \ref{geuler} admits a well-defined antidevelopment $(\gamma_\delta)_\mathscr{G}\in\Path_\mathscr{e}(\mathscr{G})$.

Furthermore, if these antidevelopments $(\gamma_\delta)_\mathscr{G}$ converge in $\Path_\mathscr{e}(\mathscr{G})$, then the antidevelopment $\gamma_\mathscr{G}\in\Path_\mathscr{e}(\mathscr{G})$ is well-defined and equal to the limit of the $(\gamma_\delta)_\mathscr{G}$ as $\delta\rightarrow 0$.\end{corollary}

If we could show that these antideveloped Euler approximation paths $(\gamma_\delta)_\mathscr{G}\in\Path_\mathscr{e}(\mathscr{G})$ always converge in $\Path_\mathscr{e}(\mathscr{G})$, then this would prove that exp-completeness implies dev-completeness, as then we would have a well-defined antidevelopment for every path in $\Path_e(G)$. Crucially, however, this convergence is not actually guaranteed, even though the derivatives are well-behaved enough that the developments converge. This is where the faulty proofs of this supposed equivalence fall apart: they erroneously assume that this convergence necessarily happens.

A similar issue arises when we try to get exp-completeness from geodesic completeness. Recall that, whenever $X,Y\in\mathfrak{g}$, we have the Trotter product formula (see Proposition 9.2.14 of \cite{HilgertNeeb2012})
\[\exp\big(t(X+Y)\big)=\lim_{k\rightarrow+\infty}\!\big(\!\exp(\tfrac{t}{k}X)\exp(\tfrac{t}{k}Y)\big)^k.\]
Analogously, whenever we have an exponential path $\gamma\in\Path_e(G)$ of the form $t\mapsto\exp(t(X+Y))$, we can try to approximate it by paths $\gamma_k\in\Path_e(G)$ given by
\[\gamma_k(t):=(\exp(\tfrac{1}{k}X)\exp(\tfrac{1}{k}Y))^{\lfloor kt\rfloor}\exp((t-\lfloor kt\rfloor/k)X)\exp((t-\lfloor kt\rfloor/k)Y),\] a composition of exponential paths from $X$ and $Y$.

For $(G,H)$ reductive, every element of $\mathfrak{g}$ can be written as a sum of some $X\in\mathfrak{m}$ and some $Y\in\mathfrak{h}$. The Trotter product formula then tells us that every exponential path in $\Path_e(G)$ can be approximated arbitrarily closely by a composition of exponential paths coming from elements of $\mathfrak{m}$ and $\mathfrak{h}$ separately. But for a Cartan geometry $(\mathscr{G},\omega)$ of type $(G,H)$, the exponential curves $t\mapsto\exp(t\omega^{-1}(Y))\mathscr{g}=\mathscr{g}\exp(tY)$ for $Y\in\mathfrak{h}$ are automatically well-defined for all $t\in\mathbb{R}$, and $(\mathscr{G},\omega)$ is defined to be geodesically complete if and only if the exponential curves of the form $t\mapsto\exp(t\omega^{-1}(X))\mathscr{g}$ for each $X\in\mathfrak{m}$ and $\mathscr{g}\in\mathscr{G}$ are well-defined for all $t\in\mathbb{R}$. Thus, if $(\mathscr{G},\omega)$ is geodesically complete, then the approximating paths $\gamma_k\in\Path_e(G)$ must have well-defined antidevelopments $(\gamma_k)_\mathscr{G}\in\Path_\mathscr{e}(\mathscr{G})$, and if those antidevelopments converge, then they must converge to $\gamma_\mathscr{G}:t\mapsto\exp(t\omega^{-1}(X+Y))\mathscr{e}$.

\begin{proposition}For $(G,H)$ reductive, suppose $(\mathscr{G},\omega)$ is a Cartan geometry of type $(G,H)$ and $\gamma\in\Path_e(G)$ is an exponential path of the form $t\mapsto\exp(t(X+Y))$ for $X\in\mathfrak{m}$ and $Y\in\mathfrak{h}$. Define $\gamma_k\in\Path_e(G)$ as above, so that $\gamma=\lim_{k\rightarrow+\infty}\gamma_k$. If $(\mathscr{G},\omega)$ is geodesically complete, then each $\gamma_k$ has a well-defined antidevelopment $(\gamma_k)_\mathscr{G}\in\Path_\mathscr{e}(\mathscr{G})$.

Furthermore, if these antidevelopments converge in $\Path_\mathscr{e}(\mathscr{G})$, then the antidevelopment $\gamma_\mathscr{G}\in\Path_\mathscr{e}(\mathscr{G})$ is well-defined and equal to the limit of the antidevelopments $(\gamma_k)_\mathscr{G}$.
\end{proposition}

Like with the case of exp-completeness and dev-completeness above, notice that the existence of these antidevelopments $(\gamma_k)_\mathscr{G}\in\Path_\mathscr{e}(\mathscr{G})$ is not enough alone to guarantee that $\gamma_\mathscr{G}:t\mapsto\exp(t\omega^{-1}(X+Y))\mathscr{e}$ is well-defined, because we would still need some way to show that the approximating antidevelopments always converge, which turns out to not be true in general. Erroneously assuming this convergence is, again, where faulty proofs of an equivalence between geodesic completeness and exp-completeness fall apart.

\subsection{Conditions for successful path approximation}
As we have seen in the previous subsection, the fact that the development map $(\cdot)_G:\Path_\mathscr{e}(\mathscr{G})\to\Path_e(G)$ is a continuous embedding tells us that the idea of approximating a path $\gamma$ by antideveloping approximations of its development $\gamma_G$ works whenever $\gamma$ actually exists; the idea fails if and only if the approximating antidevelopments fail to converge to something. Therefore, if we can find conditions that guarantee that the approximating antidevelopments converge, then the idea is still viable under those conditions.

By the usual existence and uniqueness results for ODEs, there is always some $\tau\in(0,1]$ for which $\gamma_\mathscr{G}|_{[0,\tau)}$ is well-defined, so we should expect $\gamma_\mathscr{G}$ to be well-defined if and only if it does not ``escape'' $\mathscr{G}$ before it reaches $\gamma_\mathscr{G}(1)$. In particular, we should expect our antidevelopment to be well-defined when all of the approximating paths are trapped in a compact subspace.

\begin{lemma}\label{cpctconv}
Suppose $(\gamma_k)$ is a sequence in $\Path_\mathscr{e}(\mathscr{G})$ such that the images $\gamma_k([0,1])$ are all contained in a compact subset $C\subseteq\mathscr{G}$. If the developments $(\gamma_k)_G$ converge to some $(\gamma_\infty)_G\in\Path_e(G)$, then $(\gamma_k)$ converges to a path $\gamma_\infty\in\Path_\mathscr{e}(\mathscr{G})$ with development $(\gamma_\infty)_G$.
\end{lemma}
\begin{proof}
Suppose there is some $\tau\in(0,1]$ such that the solution curve $\gamma_\infty$ to the ODE $\gamma_\infty^*\omega=(\gamma_\infty)_G^*\omega_G$ starting at $\gamma_\infty(0)=\mathscr{e}\in\mathscr{G}$ is defined on $[0,\tau)$. Each point of $\gamma_\infty|_{[0,\tau)}$ is a limit of points in $C$, so $\gamma_\infty([0,\tau))\subseteq C$. Moreover, since $C$ is compact, an $\omega$-constant metric restricted to $C$ will be complete, so since the length of $\gamma_\infty|_{[0,\tau)}$ is bounded above by $\|\omega_G(\dot{\gamma}_\infty)_G\|_{L^1}$, $\gamma_\infty|_{[0,\tau)}$ continuously extends to $\tau$, beyond which we can further extend it by using the unique flow if $\tau<1$. As such, $\gamma_\infty$ is a well-defined antidevelopment of $(\gamma_\infty)_G\in\Path_e(G)$.\mbox{\qedhere}
\end{proof}

It is also, perhaps, worth mentioning that there is an even easier criterion to guarantee that path approximation via development works, though it is exactly the criterion we would generally want to avoid using in this context.

\begin{lemma}
Suppose $(\gamma_k)$ is a sequence in $\Path_\mathscr{e}(\mathscr{G})$, where $(\mathscr{G},\omega)$ is development-complete. If the developments $(\gamma_k)_G$ converge to some $(\gamma_\infty)_G\in\Path_e(G)$, then $(\gamma_k)$ converges to a path $\gamma_\infty\in\Path_\mathscr{e}(\mathscr{G})$ with development $(\gamma_\infty)_G$.
\end{lemma}
\begin{proof}
If $(\mathscr{G},\omega)$ is development-complete, then the development map $(\cdot)_G:\Path_\mathscr{e}(\mathscr{G})\to\Path_e(G)$ is a homeomorphism. Therefore, if $(\gamma_k)_G$ converges to $(\gamma_\infty)_G$ in $\Path_e(G)$, then $\gamma_k=(\cdot)_G^{-1}((\gamma_k)_G)$ must also converge to $\gamma_\infty=(\cdot)_G^{-1}((\gamma_\infty)_G)$ in $\Path_\mathscr{e}(\mathscr{G})$.\mbox{\qedhere}
\end{proof}

\section{Clifton's examples}\label{clifton}
Before describing our main results in the next section, we think it would be instructive and helpful to the reader to review the two existing examples, due to Clifton in \cite{Clifton1966}, of Cartan geometries that are known to be complete in one way but incomplete in a stronger way. We have taken to calling these two examples the \textit{Clifton half-cylinder/can} and \textit{Clifton plane}.

\subsection{The Clifton half-cylinder/can}
The Clifton half-cylinder 
is an example of a Cartan geometry of type $(\mathbb{R}^2,\{0\})$ that is exp-complete but not dev-complete. It is defined on $\mathbb{R}/\mathbb{Z}\times(0,+\infty)$, viewed as a principal $\{0\}$-bundle over itself, with Cartan connection given by \[(\upsilon_\text{cyl})_{(x+\mathbb{Z},y)}:=\begin{bmatrix}\cos(\tfrac{1}{y}) & -\sin(\tfrac{1}{y}) \\ \sin(\tfrac{1}{y}) & \cos(\tfrac{1}{y})\end{bmatrix}\begin{bmatrix}\mathrm{d}x \\ \mathrm{d}y\end{bmatrix}\] at each point $(x+\mathbb{Z},y)\in\mathbb{R}/\mathbb{Z}\times(0,+\infty)$. This presentation of the geometry differs slightly from the original in \cite{Clifton1966}, which uses a punctured plane instead; we think that the picture on the cylinder is a bit clearer.

To start, note that $(\mathbb{R}/\mathbb{Z}\times(0,+\infty),\upsilon_\text{cyl})$ cannot be $\upsilon_\text{cyl}$-constant metric complete, since the Euclidean metric on $\mathbb{R}/\mathbb{Z}\times(0,+\infty)$ given by $\mathrm{d}x^{\otimes 2}+\mathrm{d}y^{\otimes 2}$ is an $\upsilon_\text{cyl}$-constant metric that is not complete (as a metric space). By Theorem \ref{devcompismetcomp}, it follows that the Clifton half-cylinder is not development-complete.

Despite this, the Clifton half-cylinder manages to be exp-complete. To see this, notice that all of the exponential curves for the geometry have $y$-components bounded away from $0$: since $1/y$ grows arbitrarily large as $y$ approaches $0$, any nontrivial geodesic in the geometry must turn as it moves in the direction of $y=0$ so that it approaches---but never passes---a particular curve of the form $t\mapsto(t+\mathbb{Z},y_\text{bound})$ for some $y_\text{bound}>0$. A sketch of such a geodesic is provided in Figure \ref{cylindergeodesics}. Since the $(x,y)$-coordinate expressions for the $\upsilon_\text{cyl}$-constant vector fields are all Lipschitz when restricted to $\mathbb{R}/\mathbb{Z}\times[y_\text{bound},+\infty)$, each geodesic in the geometry must be defined for all time.

\begin{figure}
\centering\includegraphics[width=0.45\textwidth]{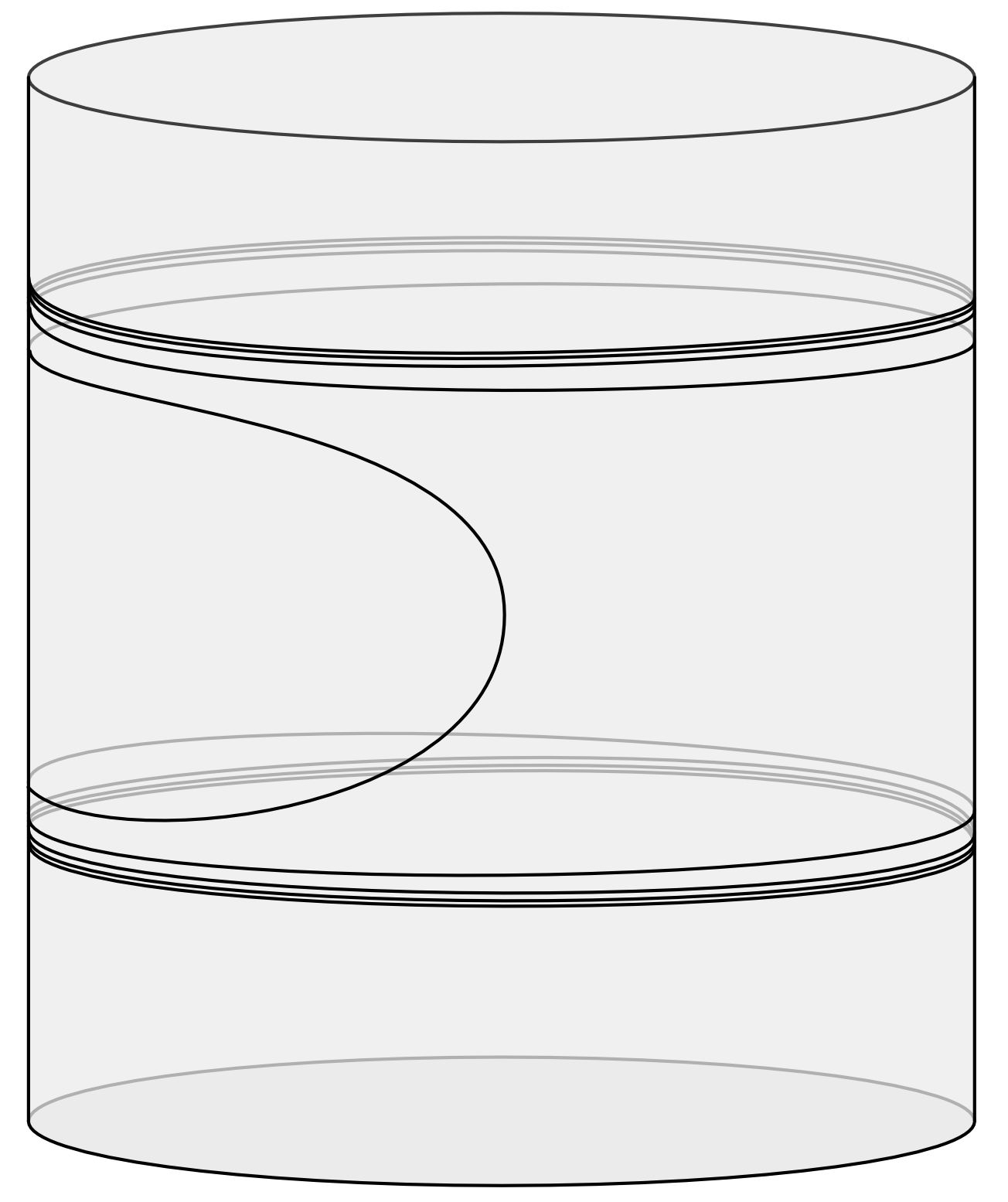}
\caption{A sketch of a geodesic in the Clifton half-cylinder}
\label{cylindergeodesics}
\end{figure}


Note that there is also a natural extension of the Clifton half-cylinder to a Cartan geometry of type $(\Iso(2),\Orth(2))$, with the trivial principal\linebreak $\Orth(2)$-bundle $(\mathbb{R}/\mathbb{Z}\times(0,+\infty))\times\Orth(2)$ and Cartan connection $\omega$ given by $\omega_{(x,y,A)}:=A^{-1}\upsilon_\text{cyl}+\MC{\Orth(2)}$. This extension is still geodesically complete, since its geodesics remain the same on the base manifold, but it also still cannot be development-complete. In particular, this tells us that the Hopf--Rinow theorem does not give an equivalence between the three notions of completeness for \textit{arbitrary} Cartan geometries of type $(\Iso(m),\Orth(m))$: it only applies to those with vanishing torsion, meaning that the projection of the curvature $\Omega:=\mathrm{d}\omega+\tfrac{1}{2}[\omega,\omega]$ to $\mathbb{R}^m$ must vanish for it to work.

\subsection{The Clifton plane}
The Clifton plane is the translation-invariant Cartan geometry $(\Aff(2),\MC{\Aff(2)}\!+\phi)$ of type $(\Aff(2),\GLin_2\mathbb{R})$ over the plane $\mathbb{R}^2\cong\Aff(2)/\GLin_2\mathbb{R}$, determined by the $\mathfrak{gl}_2\mathbb{R}$-valued deformation $\phi:T(\Aff(2))\to\mathfrak{gl}_2\mathbb{R}$ given by \[\phi_{([\begin{smallmatrix} x \\ y\end{smallmatrix}],A)}:=\Ad_{A^{-1}}\begin{bmatrix}0 & -\mathrm{d}y \\ \mathrm{d}y & \mathrm{d}x\end{bmatrix}\] for each $([\begin{smallmatrix} x \\ y\end{smallmatrix}],A)\in\Aff(2)\simeq\mathbb{R}^2\rtimes\GLin_2\mathbb{R}$.

We have
\[(\MC{\Aff(2)}+\phi)^{-1}(v,R)=\MC{\Aff(2)}^{-1}\Big(v,R-\phi(\MC{\Aff(2)}^{-1}(v))\Big),\]
so in particular,
\begin{align*}\exp\Big(t(\MC{\Aff(2)}+\phi)^{-1}([\begin{smallmatrix} 1 \\ 0\end{smallmatrix}])\Big)(0,\mathds{1}) & =\exp\Big(t\MC{\Aff(2)}^{-1}([\begin{smallmatrix} 1 \\ 0\end{smallmatrix}],[\begin{smallmatrix} 0 & 0 \\ 0 & -1\end{smallmatrix}])\Big)(0,\mathds{1}) \\ & =\left(\left[\begin{smallmatrix} t \\ 0\end{smallmatrix}\right],\left[\begin{smallmatrix} 1 & 0 \\ 0 & \mathrm{e}^{-t}\end{smallmatrix}\right]\right).\end{align*}
Finding $\exp(t(\MC{\Aff(2)}+\phi)^{-1}([\begin{smallmatrix} 0 \\ 1\end{smallmatrix}]))(0,\mathds{1})$ is a bit trickier, but by a morass of not especially enlightening computation, we can determine that the image on the base manifold of this latter exponential curve is
\[t\mapsto\begin{bmatrix}\tfrac{1}{2}\log(\cosh(\sqrt{2}t)) \\ \sqrt{2}\arctan(\mathrm{e}^{\sqrt{2}t})-\tfrac{\pi}{2\sqrt{2}}\end{bmatrix}.\]
We have sketched the image of this geodesic in Figure \ref{planegeodesics}. The key thing to notice is that, if we take each velocity of this geodesic on the base manifold and apply a translation to push it back into the tangent space at $0\in\mathbb{R}^2$, then we get a nonzero element of every one-dimensional subspace of $T_0\mathbb{R}^2$ except for $\langle\partial_x\rangle$, which we can get from the velocity for the geodesic corresponding to $[\begin{smallmatrix} 1 \\ 0\end{smallmatrix}]$ through $(0,\mathds{1})$. In other words, because $\MC{\Aff(2)}+\phi$ is invariant under translations and geodesics on the base manifold are uniquely determined by their velocity at a single point, every geodesic on the Clifton plane is a translation of one of these two geodesics, up to affine reparametrization. Consequently, it follows that the Clifton plane is geodesically complete.

\begin{figure}
\centering\includegraphics[width=0.5\textwidth]{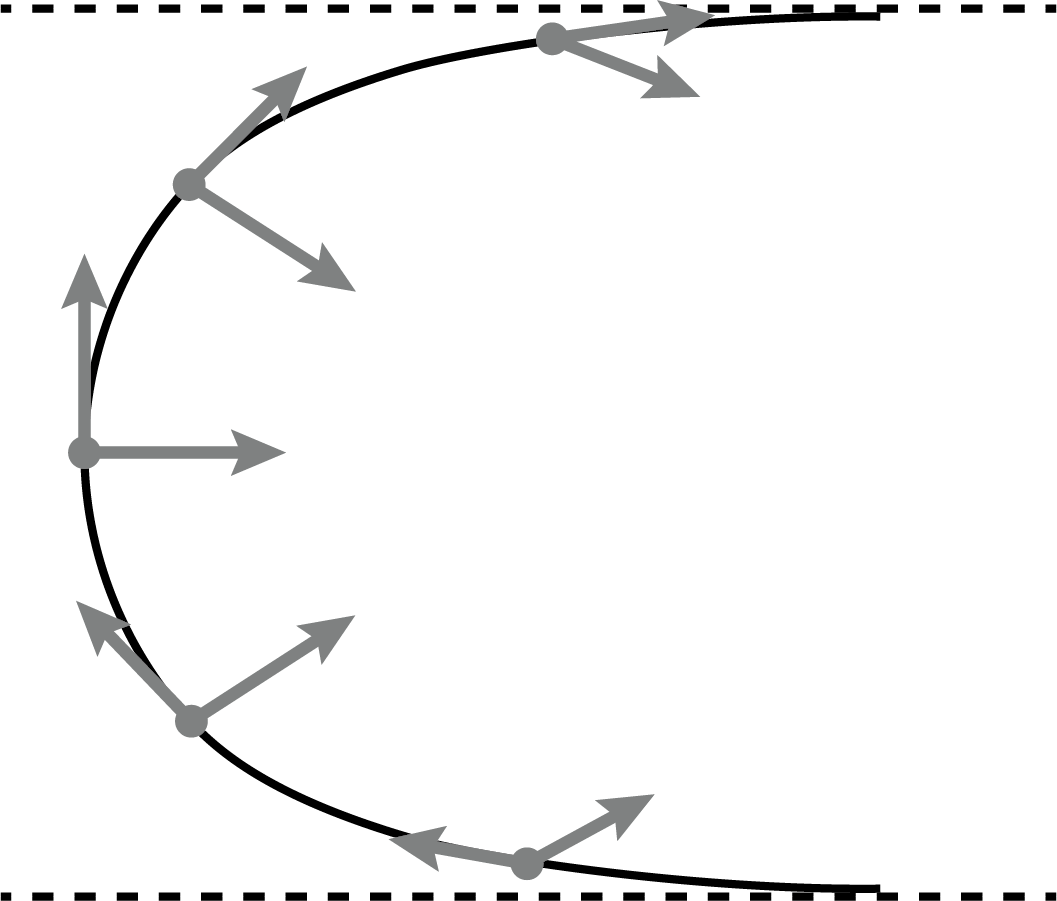}
\caption{A sketch of a geodesic over the Clifton plane whose image is not contained in a line}
\label{planegeodesics}
\end{figure}

While affine lines in $\mathbb{R}^2$ that are not of the form $t\mapsto a[\begin{smallmatrix}t \\ 0\end{smallmatrix}]+[\begin{smallmatrix}b \\ c\end{smallmatrix}]$ for some constants $a,b,c\in\mathbb{R}$ are no longer geodesics for the Clifton plane, they are still exponential curves: for each $v\in\mathbb{R}^2<\mathfrak{aff}(2)$, we have $(\MC{\Aff(2)}+\phi)(\MC{\Aff(2)}^{-1}(v))=v+\phi(\MC{\Aff(2)}^{-1}(v))$, which is constant along the flow of $\MC{\Aff(2)}^{-1}(v)$. In particular, we have \begin{align*}\left(t[\begin{smallmatrix}v_1 \\ v_2\end{smallmatrix}],\mathds{1}\right) & =\exp\big(t([\begin{smallmatrix}v_1 \\ v_2\end{smallmatrix}],0)\big)=\exp\big(t\MC{\Aff(2)}^{-1}([\begin{smallmatrix}v_1 \\ v_2\end{smallmatrix}],0)\big)(0,\mathds{1}) \\ & =\exp\Big(t(\MC{\Aff(2)}\!+\phi)^{-1}\big([\begin{smallmatrix}v_1 \\ v_2\end{smallmatrix}],[\begin{smallmatrix}0 & -v_2 \\ v_2 & v_1\end{smallmatrix}]\big)\Big)(0,\mathds{1}).\end{align*}

Let us specifically consider the vector $[\begin{smallmatrix}3 \\ \sqrt{2}\end{smallmatrix}]\in\mathbb{R}^2<\mathfrak{aff}(2)$. We can compute that
\[\phi_{([\begin{smallmatrix}x \\ y\end{smallmatrix}],\mathds{1})}\left(\MC{\Aff(2)}^{-1}\big([\begin{smallmatrix}3 \\ \sqrt{2}\end{smallmatrix}]\big)\right)=\left[\begin{smallmatrix}0 & -\sqrt{2} \\ \sqrt{2} & 3\end{smallmatrix}\right],\]
which has eigenvectors $[\begin{smallmatrix}\sqrt{2} \\ -1\end{smallmatrix}]$ and $[\begin{smallmatrix}1 \\ -\sqrt{2}\end{smallmatrix}]$ with corresponding eigenvalues $1$ and $2$, respectively. For $v\in\mathbb{R}^2$ and invertible $R\in\mathfrak{gl}_2\mathbb{R}$, we have
\[\exp\!\big(t(v,R)\big)=\big((\exp(tR)-\mathds{1})R^{-1}v,\exp(tR)\big),\]
so since $4\sqrt{2}[\begin{smallmatrix}\sqrt{2} \\ -1\end{smallmatrix}]-5[\begin{smallmatrix}1 \\ -\sqrt{2}\end{smallmatrix}]=[\begin{smallmatrix}3 \\ \sqrt{2}\end{smallmatrix}]$, the development $\gamma_{\Aff(2)}$ (with respect to ${\MC{\Aff(2)}\!+\phi}$) of the curve $\gamma:\mathbb{R}\to\Aff(2)$ given by $t\mapsto\big(t[\begin{smallmatrix}3 \\ \sqrt{2}\end{smallmatrix}],\mathds{1}\big)$ satisfies
\begin{align*}\quot{\GLin_2\mathbb{R}}(\gamma_{\Aff(2)}(t)) & =\gamma_{\Aff(2)}(t)\cdot 0 = \exp\left(t\left(\big[\begin{smallmatrix}3 \\ \sqrt{2}\end{smallmatrix}\big],\left[\begin{smallmatrix}0 & -\sqrt{2} \\ \sqrt{2} & 3\end{smallmatrix}\right]\right)\right)\cdot 0 \\ & =\left(\exp\!\left(t\left[\begin{smallmatrix}0 & -\sqrt{2} \\ \sqrt{2} & 3\end{smallmatrix}\right]\right)-\mathds{1}\right)\left[\begin{smallmatrix}0 & -\sqrt{2} \\ \sqrt{2} & 3\end{smallmatrix}\right]^{-1}\big[\begin{smallmatrix}3 \\ \sqrt{2}\end{smallmatrix}\big] \\ & =\left(\exp\!\left(t\left[\begin{smallmatrix}0 & -\sqrt{2} \\ \sqrt{2} & 3\end{smallmatrix}\right]\right)-\mathds{1}\right)(4\sqrt{2}[\begin{smallmatrix}\sqrt{2} \\ -1\end{smallmatrix}]-\tfrac{5}{2}[\begin{smallmatrix}1 \\ -\sqrt{2}\end{smallmatrix}]) \\ & =4\sqrt{2}\mathrm{e}^t\big[\begin{smallmatrix}\sqrt{2} \\ -1\end{smallmatrix}\big]-\tfrac{5\mathrm{e}^{2t}}{2}\big[\begin{smallmatrix}1 \\ -\sqrt{2}\end{smallmatrix}\big]-\tfrac{1}{2}\big[\begin{smallmatrix}11 \\ -3\sqrt{2}\end{smallmatrix}\big].\end{align*}
Notably, the image $\quot{\GLin_2\mathbb{R}}(\gamma_{\Aff(2)}(\mathbb{R}))$ of this projected curve is properly contained in the parabola $t\mapsto 4\sqrt{2}t\big[\begin{smallmatrix}\sqrt{2} \\ -1\end{smallmatrix}\big]-\tfrac{5t^2}{2}\big[\begin{smallmatrix}1 \\ -\sqrt{2}\end{smallmatrix}\big]-\tfrac{1}{2}\big[\begin{smallmatrix}11 \\ -3\sqrt{2}\end{smallmatrix}\big]$, as illustrated in Figure \ref{halfparabola}.

\begin{figure}[h]
\centering\includegraphics[width=0.8\textwidth]{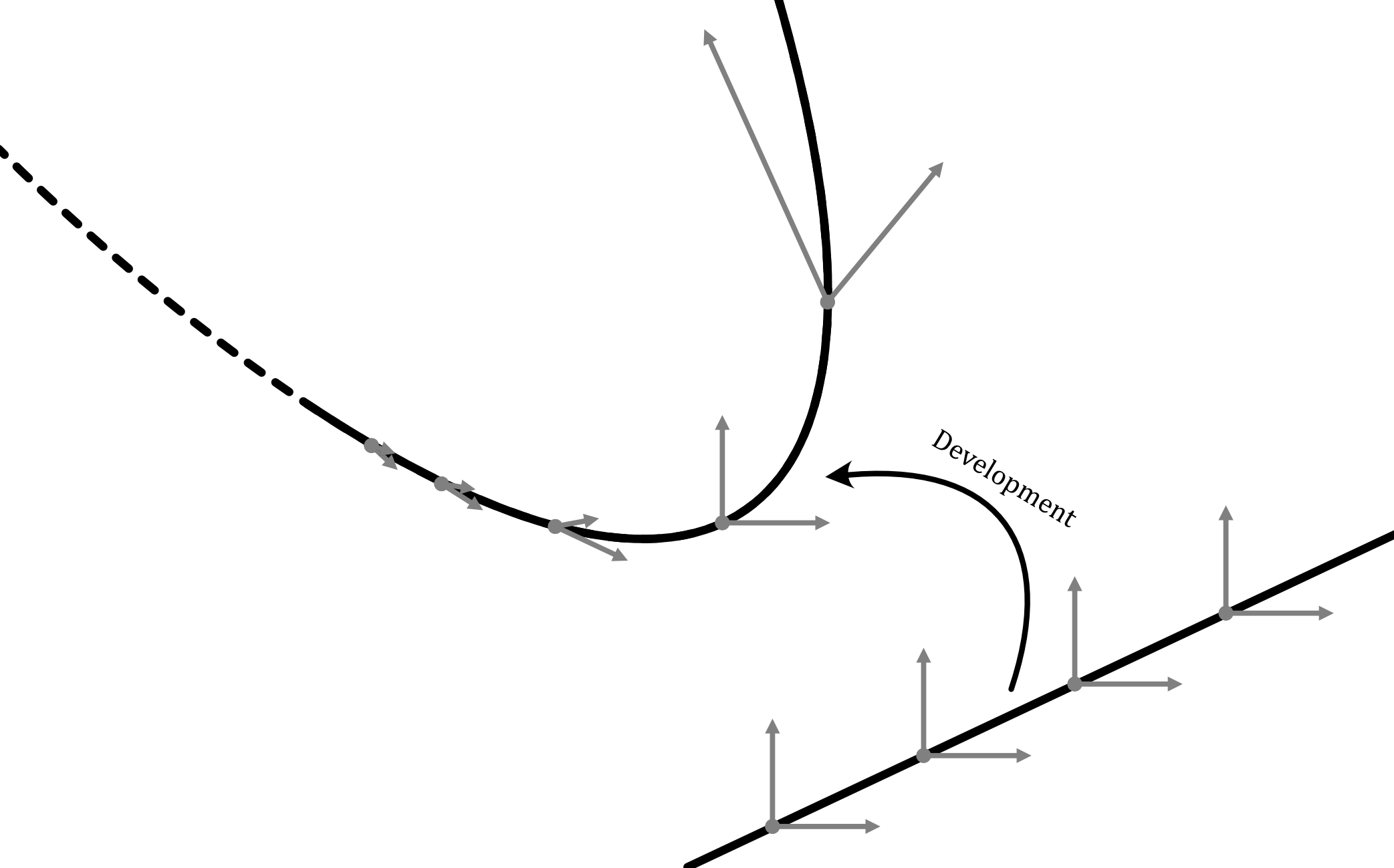}
\caption{A sketch of the development from the Clifton plane of a line in $\mathbb{R}^2$ with velocity $[\begin{smallmatrix}3 \\ \sqrt{2}\end{smallmatrix}]$, which is properly contained in the parabola indicated by the dotted curve}
\label{halfparabola}
\end{figure}

Look, however, at the one-parameter subgroup $\tilde{\gamma}_{\Aff(2)}:\mathbb{R}\to\Aff(2)$ given by
\begin{align*}\tilde{\gamma}_{\Aff(2)}(t) & :=\exp\!\left(t\!\left(\big[\begin{smallmatrix}3 \\ \sqrt{2}\end{smallmatrix}\big],\tfrac{-5\sqrt{2}}{8}\!\left[\begin{smallmatrix}\sqrt{2} & 1 \\ -2 & -\sqrt{2}\end{smallmatrix}\right]\right)\!\right) \\ & =\!\left(t\big[\begin{smallmatrix}3 \\ \sqrt{2}\end{smallmatrix}\big]\!-\!\tfrac{5t^2}{2}\!\big[\begin{smallmatrix}1 \\ -\sqrt{2}\end{smallmatrix}\big],\mathds{1}\!-\!\tfrac{5\sqrt{2}t}{8}\!\left[\begin{smallmatrix}\sqrt{2} & 1 \\ -2 & -\sqrt{2}\end{smallmatrix}\right]\right)\!;\end{align*}
this traces out the same parabola (up to reparametrization) as $\gamma_{\Aff(2)}$ in $\mathbb{R}^2=\Aff(2)/\GLin_2\mathbb{R}$, except this time it does not stop ``halfway'' along, so we can write
\begin{align*}\quot{\GLin_2\mathbb{R}}(\tilde{\gamma}_{\Aff(2)}(\mathrm{e}^t-1)) & =(\mathrm{e}^t-1)\big[\begin{smallmatrix}3 \\ \sqrt{2}\end{smallmatrix}\big]\!-\!\tfrac{5(\mathrm{e}^t-1)^2}{2}\!\big[\begin{smallmatrix}1 \\ -\sqrt{2}\end{smallmatrix}\big] \\ & =4\sqrt{2}\mathrm{e}^t\big[\begin{smallmatrix}\sqrt{2} \\ -1\end{smallmatrix}\big]-\tfrac{5\mathrm{e}^{2t}}{2}\big[\begin{smallmatrix}1 \\ -\sqrt{2}\end{smallmatrix}\big]-\tfrac{1}{2}\big[\begin{smallmatrix}11 \\ -3\sqrt{2}\end{smallmatrix}\big] \\ & =\quot{\GLin_2\mathbb{R}}(\gamma_{\Aff(2)}(t)).\end{align*}
Thus, the exponential curve over the Clifton plane given by
\[\tilde{\gamma}:t\mapsto\exp\!\left(t(\MC{\Aff(2)}\!+\phi)^{-1}\left(\big[\begin{smallmatrix}3 \\ \sqrt{2}\end{smallmatrix}\big],\tfrac{-5\sqrt{2}}{8}\!\left[\begin{smallmatrix}\sqrt{2} & 1 \\ -2 & -\sqrt{2}\end{smallmatrix}\right]\right)\right)(0,\mathds{1})\]
can only be defined on $(-1,+\infty)$, since $\quot{\GLin_2\mathbb{R}}(\tilde{\gamma}(t))=\log(t+1)\big[\begin{smallmatrix}3 \\ \sqrt{2}\end{smallmatrix}\big]$, and in particular, the Clifton plane is not exp-complete.



\section{Generalizing the Clifton plane construction}\label{halfsteady}
We would like to build more examples like the Clifton plane, where we get a reductive Cartan geometry that is geodesically complete but not exponential-complete. To do this, we will first untangle how the Clifton plane manages to break exponential-completeness while still maintaining geodesic completeness, and then generalize these tools to more general reductive Cartan geometries.

\subsection{Half-fine pairs}\label{halfinated}
Before, when we showed that the Clifton plane is not exp-complete, we seemed to stumble into the fact that there is an inextensible exponential curve by blind luck. To build more examples, we will need to figure out how to find or force such exponential curves in other situations.

To help visualize where such exponential curves can come from, let us use 1-dimensional affine geometry $(\Aff(1),\mathbb{R}^\times)$ as a toy model, where $\Aff(1)\simeq\mathbb{R}\rtimes\mathbb{R}^\times$. For $X=(1,0)\in\mathfrak{aff}(1)$ and $Y=(0,-1)\in\mathfrak{aff}(1)$, consider the exponential curves
\[\gamma:\mathbb{R}\to\Aff(1),~t\mapsto\exp(tX)=(t,1)\]
and
\[\gamma^{+Y}:\mathbb{R}\to\Aff(1),~t\mapsto\exp(t(X+Y))=(1-\mathrm{e}^{-t},\mathrm{e}^{-t}).\]
Here, we can think of $\gamma$ as the usual linear parametrization of the affine line, viewed as translations in $\Aff(1)$, whereas $\gamma^{+Y}$ is the unique curve with $\gamma^{+Y}(0)=\gamma(0)=(0,1)$ and $\MC{\Aff(1)}(\dot{\gamma}^{+Y})=\MC{\Aff(1)}(\dot{\gamma})+Y$. This alteration to the velocity of $\gamma$ causes $\gamma^{+Y}$ to slow down more and more, so that the underlying curve $\quot{\mathbb{R}^\times}(\gamma^{+Y}):t\mapsto 1-\mathrm{e}^{-t}$ on the base manifold $\Aff(1)/\mathbb{R}^\times\cong\mathbb{R}$ converges to $1$ as $t\rightarrow+\infty$.

Now, let $\phi$ be the $\mathfrak{aff}(1)$-valued deformation of $\MC{\Aff(1)}$ determined by setting $\phi(\dot{\gamma})=Y$. For the new Cartan connection $\MC{\Aff(1)}+\phi$ on $\Aff(1)$ over $\Aff(1)/\mathbb{R}^\times\cong\mathbb{R}$, we have that
\[(\MC{\Aff(1)}+\phi)(\dot{\gamma})=X+Y=\MC{\Aff(1)}(\dot{\gamma}^{+Y}),\]
so the development of $\gamma$ with respect to $\MC{\Aff(1)}+\phi$ is $\gamma^{+Y}$; we include a visual depiction of this in Figure \ref{aff1example}. Notably, $\gamma^{+Y}$ cannot go past $1$ on the base manifold $\mathbb{R}$, so if we try to antidevelop a curve that does go past $1$, then that antidevelopment in $(\Aff(1),\MC{\Aff(1)}+\phi)$ needs to ``run off the edge'' of $\mathbb{R}$. The exponential curve for $X=\MC{\Aff(1)}(\dot{\gamma})$ with respect to $\MC{\Aff(1)}+\phi$, for example, is given by
\[t\mapsto\exp\left(t\left(\MC{\Aff(1)}\!+\phi\right)^{-1}(X)\right)(0,1)=\left(\log(\tfrac{1}{1-t}),\frac{1}{1-t}\right),\]
which escapes to infinity as $t\rightarrow 1$. In particular, $(\Aff(1),\MC{\Aff(1)}+\phi)$ is not exponential-complete.

\begin{figure}[h]
\centering\includegraphics[width=0.8\textwidth]{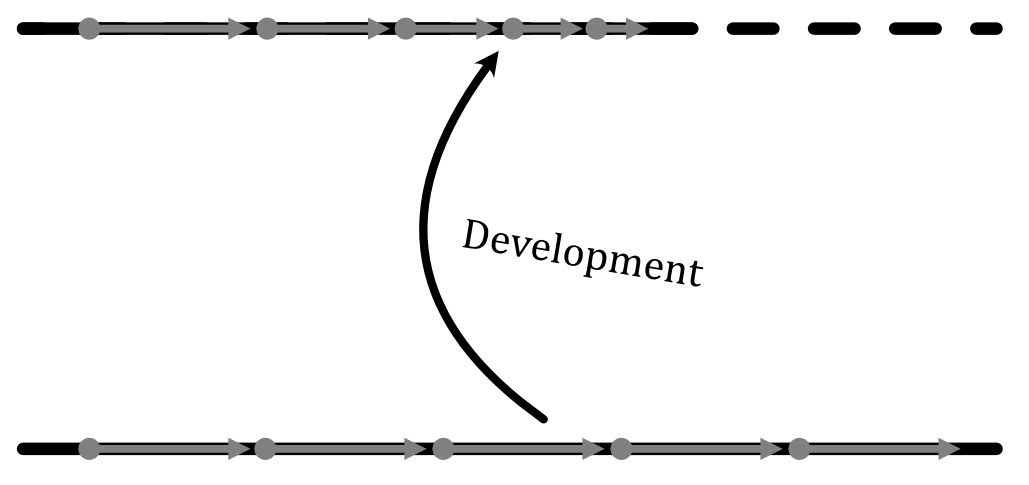}
\caption{A sketch of the curve $\gamma:t\mapsto(t,1)$ in $\Aff(1)$ and its development $\gamma^{+Y}$ with respect to the deformed Cartan connection $\MC{\Aff(1)}+\phi$}
\label{aff1example}
\end{figure}

The sketch of the exponential curve $t\mapsto(t,1)$ and its development with respect to a deformed connection in Figure \ref{aff1example} should, hopefully, be vaguely reminiscent of the sketch of the development of the line in the Clifton plane from Figure \ref{halfparabola}. This is intentional, as they both essentially amount to deformations of 1-dimensional affine geometry. To see this for the Clifton plane, let $X=\big[\begin{smallmatrix}3 \\ \sqrt{2}\end{smallmatrix}\big]$, $Y=\big[\begin{smallmatrix}0 & -\sqrt{2} \\ \sqrt{2} & 3\end{smallmatrix}\big]$, and $Z=\tfrac{-5\sqrt{2}}{8}\!\left[\begin{smallmatrix}\sqrt{2} & 1 \\ -2 & -\sqrt{2}\end{smallmatrix}\right]$. The development of the line $\gamma:t\mapsto\exp(tX)$ in the Clifton plane is precisely $\gamma^{+Y}:t\mapsto\exp(t(X+Y))$, and the exponential curve lying over the parabola whose image in the base manifold contains $\quot{\GLin_2\mathbb{R}}(\gamma^{+Y}(\mathbb{R}))$ is $t\mapsto\exp(t(X+Z))$. Notice that \[[Y-Z,X+Z]=[Y-Z,X]+[Y,Z]=X+Z,\] so we have an isomorphism of Lie algebras $\langle X+Z,Y-Z\rangle\approx\mathfrak{aff}(1)$ given by $a(X+Z)+s(Y-Z)\mapsto(a,s)$. Aside from the fact that the example in the Clifton plane is parametrized to move in the opposite direction to the example in the affine line above, this isomorphism shows us that these two cases behave in effectively the same way, which leads us to a general method of building deformations of Cartan geometries to get examples that are not exp-complete.

\begin{definition}
In a model $(G,H)$, a pair $(Y,Z)\in\mathfrak{h}\times\mathfrak{h}$ is called a \emph{half-fine pair} for $X\in\mathfrak{g}\setminus\mathfrak{h}$ if and only if $[Z-Y,X+Z]=\lambda(X+Z)$ for some $\lambda>0$, so that $\langle X+Z,Z-Y\rangle$ is isomorphic to $\mathfrak{aff}(1)\approx\mathbb{R}\niplus\mathbb{R}$.
\end{definition}

Crucially, since $[Z-Y,X+Z]=\lambda(X+Z)$,
\[\exp(a(X+Z)+s(Z-Y))=\exp\left(a\tfrac{\mathrm{e}^{\lambda s}-1}{\lambda s}(X+Z)\right)\exp(s(Z-Y))\]
for $s\neq 0$, so the image of the curve
\[t\mapsto \quot{H}(\exp(t(X+Y)))=\quot{H}(\exp(t((X+Z)-(Z-Y))))\]
is properly contained in the image of $t\mapsto q_{{}_H}(\exp(t(X+Z)))$ because $\quot{H}(\exp(\mathbb{R}(X+Y)))=\quot{H}(\exp((-\infty,\frac{1}{\lambda})(X+Z)))$. Hence, we can think of the exponential curve for $X+Y$ as only going ``halfway'' along the exponential curve for $X+Z$ when we project to the base manifold.

\begin{lemma}\label{halfnotcomplete}
Suppose $(\mathscr{G},\omega)$ is a Cartan geometry of type $(G,H)$ over $M$ and $\gamma:\mathbb{R}\to\mathscr{G},~t\mapsto\exp(t\omega^{-1}(X))\mathscr{e}$ is an exponential curve such that $\quot{H}(\gamma(t))$ does not converge as $t\rightarrow+\infty$. If $(Y,Z)$ is a half-fine pair for $X$ and $\phi$ is a deformation of $\omega$ such that $\phi(\dot{\gamma})=Y$, then $(\mathscr{G},\omega+\phi)$ is not exp-complete.
\end{lemma}
\begin{proof}
By our choice of $\phi$, the development of $\gamma$ with respect to $\omega+\phi$ is given by $t\mapsto\exp(t(X+Y))$. Thus, $\gamma(t)=\exp(t(\omega+\phi)^{-1}(X+Y))\mathscr{e}$, and since $(Y,Z)$ must satisfy $[Z-Y,X+Z]=\lambda(X+Z)$ for some $\lambda>0$, the curve $t\mapsto\exp(t(\omega+\phi)^{-1}(X+Z))\mathscr{e}$ can only be defined for $t<\frac{1}{\lambda}$ because
\[\quot{H}(\gamma(s))=\quot{H}\left(\exp\left(\frac{1-\mathrm{e}^{-\lambda s}}{\lambda}(\omega+\phi)^{-1}(X+Z)\right)\mathscr{e}\right),\] hence $(\omega+\phi)^{-1}(X+Z)$ is not complete.\mbox{\qedhere}
\end{proof}

\subsection{Steadiness}
In addition to a way to build something that breaks exponential-completeness, we also need a way to guarantee that the resulting geometry stays geodesically complete. For this purpose, we draw inspiration from the Riemannian case, where completeness is a relatively nice assumption.

In the Clifton plane, notice that all of the geodesics on the base manifold $\mathbb{R}^2$ have constant speed with respect to the Riemannian metric $\mathrm{g}=2\mathrm{d}x^{\otimes 2}+\mathrm{d}y^{\otimes 2}$: for $\gamma_1:t\mapsto \begin{smallbmatrix}t \\ 0\end{smallbmatrix}$, $\mathrm{g}(\dot{\gamma}_1,\dot{\gamma}_1)=2$, while for
\[\gamma_2:t\mapsto\begin{bmatrix}\tfrac{1}{2}\log(\cosh(\sqrt{2}t)) \\ \sqrt{2}\arctan(\mathrm{e}^{\sqrt{2}t})-\tfrac{\pi}{2\sqrt{2}}\end{bmatrix},\]
we have
\[\mathrm{g}(\dot{\gamma}_2(t),\dot{\gamma}_2(t))=2(\tfrac{\tanh(\sqrt{2}t)}{\sqrt{2}})^2+(\operatorname{sech}(\sqrt{2}t))^2=1,\]
with all of the other geodesics being affine reparametrizations of these up to translation. Therefore, if we did not already know that $\gamma_1$ and $\gamma_2$ are defined on all of $\mathbb{R}$, then we could use the completeness of $\mathbb{R}^2$ with respect to $\mathrm{g}$ to obtain this, as $\gamma_i$ restricted to some finite interval would have finite length and hence a continuous extension to its endpoints. This suggests a condition we can use for other reductive geometries as well.

\begin{definition}Let $\mathrm{g}$ be a Riemannian metric on $M$. We say that a Cartan geometry $(\mathscr{G},\omega)$ of reductive type $(G,H)$ over $M$ is \emph{$\mathrm{g}$-steady} if and only if $\omega^{-1}(X)\,\quot{H}^*\mathrm{g}(\omega^{-1}(X),\omega^{-1}(X))=0$ for all $X\in\mathfrak{m}$.\end{definition}

Being $\mathrm{g}$-steady amounts to all of the geodesics for $(\mathscr{G},\omega)$ on the base manifold $M$ having constant speed with respect to $\mathrm{g}$: for $X\in\mathfrak{m}$, $\gamma(t)=\exp(t\omega^{-1}(X))\mathscr{e}$, and $\bar{\gamma}=\quot{H}(\gamma)$,
\[\partial_t\mathrm{g}(\dot{\bar{\gamma}},\dot{\bar{\gamma}})=\partial_t\quot{H}^*\mathrm{g}(\dot{\gamma},\dot{\gamma})=\omega^{-1}(X)\,\quot{H}^*\mathrm{g}(\omega^{-1}(X),\omega^{-1}(X)).\]
In particular, when $(M,\mathrm{g})$ is metrically complete, we can always extend geodesics to be defined on all of $\mathbb{R}$.

\begin{lemma}\label{slowandsteady} Let $(\mathscr{G},\omega)$ be a reductive Cartan geometry of type $(G,H)$ over $M$. If $(M,\mathrm{g})$ is a complete Riemannian manifold and $(\mathscr{G},\omega)$ is $\mathrm{g}$-steady, then $(\mathscr{G},\omega)$ is geodesically complete.\end{lemma}
\begin{proof}Suppose $\bar{\gamma}:[0,\tau)\to M$ is a base manifold geodesic for $(\mathscr{G},\omega)$. Since $\bar{\gamma}$ has constant speed with respect to $\mathrm{g}$, it has finite $\mathrm{g}$-length, so by completeness of $(M,\mathrm{g})$, we can extend $\bar{\gamma}$ to $\tau$, hence we can extend $\bar{\gamma}$ to be defined on all of $[0,+\infty)$. To get the negative half of $\mathbb{R}$, we can apply the same argument to a geodesic in the reverse direction.\mbox{\qedhere}\end{proof}

\subsection{A recipe for generalizing the Clifton plane} When we put Lemma \ref{halfnotcomplete} and Lemma \ref{slowandsteady} together, we get the following result.

\begin{theorem}\label{speedyandsteady} Suppose $(\mathscr{G},\omega)$ is a reductive Cartan geometry of type $(G,H)$ over $M$ and $\gamma:\mathbb{R}\to\mathscr{G},~t\mapsto\exp(t\omega^{-1}(X))\mathscr{e}$ is an exponential curve such that $\quot{H}(\gamma(t))$ does not converge as $t\rightarrow+\infty$. Let $(Y,Z)$ be a half-fine pair for $X$ and let $\mathrm{g}$ be a complete Riemannian metric on $M$. If there exists a deformation $\phi$ of $\omega$ such that $\phi(\dot{\gamma})=Y$ and $(\mathscr{G},\omega+\phi)$ is $\mathrm{g}$-steady, then $(\mathscr{G},\omega+\phi)$ is geodesically complete but not exp-complete.\end{theorem}

This leads to a generalization of the Clifton plane construction for other connected Lie groups of dimension at least 2.

\begin{corollary}\label{cliftongen} On each connected Lie group $S$ of dimension at least 2, there exists an $S$-invariant affine geometry that is geodesically complete but not exp-complete.\end{corollary}
\begin{proof}For each connected Lie group $S$ of dimension greater than 1, let $\rho:\mathfrak{s}\to\mathbb{R}^{\dim(S)}$ be a linear isomorphism, so that we get an $S$-invariant Cartan geometry $(S\times\GLin_{\dim(S)}\!\mathbb{R},\omega_\rho)$ of type $(\Aff(\dim(S)),\GLin_{\dim(S)}\!\mathbb{R})$ over $S$ with Cartan connection given by
\[(\omega_\rho)_{(s,A)}:=\Ad_{A^{-1}}\rho(\MC{S})+ \MC{\GLin_{\dim(S)}\!\mathbb{R}}.\]
In this geometry, for each $v\in\mathbb{R}^{\dim(S)}<\mathfrak{aff}(\dim(S))$, the corresponding exponential curves are given by
\[\exp(t\omega_\rho^{-1}(v))(s,A)=(s\exp(t\rho^{-1}(Av)),A),\]
so the underlying curves on the base manifold exist for all time and do not converge as $t\rightarrow+\infty$ whenever $v\neq 0$.

Let $\{e_1,\dots,e_{\dim(S)}\}$ be an orthonormal basis for $\mathbb{R}^{\dim(S)}$ with respect to the usual dot product, and let $\mathrm{g}$ be the inner product on $\mathbb{R}^{\dim(S)}$ given by $\mathrm{g}(e_1,e_1)=2$, $\mathrm{g}(e_i,e_i)=1$ for $i>1$, and $\mathrm{g}(e_i,e_j)=0$ whenever $i\neq j$. Consider the left-invariant Riemannian metric $(\rho^*\mathrm{g})_{\MC{S}}$ on $S$ given by
\[(\rho^*\mathrm{g})_{\MC{S}}(\xi_1,\xi_2):=\mathrm{g}(\rho(\MC{S}(\xi_1)),\rho(\MC{S}(\xi_2)))\]
and the $S$-invariant $\mathfrak{gl}_{\dim(S)}\mathbb{R}$-valued deformation $\phi$ of $\omega_\rho$ determined by setting $\phi_{(s,\mathds{1})}(\MC{S}^{-1}(\rho^{-1}(v)))$ to be the $\dim(S)\times\dim(S)$-matrix whose top-left $2\times 2$ block is given by $\begin{smallbmatrix}0 & -v_2 \\ v_2 & v_1\end{smallbmatrix}$ and for which all other entries are $0$, where $v=v_1e_1+\cdots v_{\dim(S)}e_{\dim(S)}$; when $\dim(S)=2$, this precisely corresponds to the deformation $\phi$ for the Clifton plane. Then, since
\[\mathrm{g}(w,\phi(\MC{S}^{-1}(\rho^{-1}(w)))w)=-2w_1w_2^2+w_2(w_1w_2+w_2w_1)=0,\]
$(\omega_\rho+\phi)_{(s,A)}^{-1}(v)\quot{\GLin_{\dim(S)}\mathbb{R}}^*(\rho^*\mathrm{g})_{\MC{S}}((\omega_\rho+\phi)^{-1}(v),(\omega_\rho+\phi)^{-1}(v))$ is equal to
\begin{align*}-\MC{H}^{-1}(\Ad_{A^{-1}}\phi(\MC{S}^{-1}(\rho^{-1}(Av))))\mathrm{g}(Av,Av) & =-2\mathrm{g}(Av,\phi(\MC{S}^{-1}(\rho^{-1}(Av)))Av) \\ & =0,\end{align*}
so $\omega_\rho+\phi$ is $(\rho^*\mathrm{g})_{\MC{S}}\!$-steady, hence geodesically complete. On the other hand, for $X=3e_1+\sqrt{2}e_2\in\mathbb{R}^{\dim(S)}$ and $Z$ the matrix whose entries are $0$ except for the top-left $2\times 2$ block given by $\tfrac{-5\sqrt{2}}{8}\begin{smallbmatrix}\sqrt{2} & 1 \\ -2 & -\sqrt{2}\end{smallbmatrix}$, we have that $(-\phi(\MC{S}^{-1}(\rho^{-1}(X))),-Z)$ is a half-fine pair for $-X$, just like for the Clifton plane, so $(S\times\GLin_{\dim(S)}\mathbb{R},\omega_\rho+\phi)$ is not exp-complete.\mbox{\qedhere}\end{proof}

\section{Exp-complete dev-incomplete affine examples}\label{rescalingdeformations}
We would also like to build a large family of exp-complete examples that are not dev-complete, to further highlight just how distinct these two notions of completeness are. For convenience, we will restrict to the affine case $(\Aff(m),\GLin_m\mathbb{R})$ in this section.

\subsection{Rescaling deformations}
For an affine geometry $(\mathscr{G},\omega)$ over $M$, we have a natural class of deformations of the form $\quot{\GLin_m\mathbb{R}}^*\beta\otimes\mathds{1}$, where $\mathds{1}\!\in\!\mathfrak{z}(\mathfrak{gl}_m\mathbb{R})$ is the identity matrix and $\beta$ is a (smooth) $1$-form on $M$. We will refer to these as \emph{rescaling deformations}, and denote the resulting Cartan connection on $\mathscr{G}$ by $\omega_{+\beta}:=\omega+\quot{\GLin_m\mathbb{R}}^*\beta\otimes\mathds{1}$.

As the name hopefully suggests, rescaling deformations only change the scale of the affine geometry, leaving the rest of the isotropy as it is. Since rescaling preserves lines through $0$, the geodesics of $(\mathscr{G},\omega_{+\beta})$ are, consequently, just reparametrizations of the geodesics of $(\mathscr{G},\omega)$ after right-translating by an appropriate curve in $\mathbb{R}^\times\mathds{1}<\GLin_m\mathbb{R}$. More specifically, for an element $v\in\mathfrak{m}=\mathbb{R}^m<\mathfrak{aff}(m)$, consider the geodesic $\gamma:t\mapsto\exp(t\omega^{-1}(v))\mathscr{e}$ for $(\mathscr{G},\omega)$. Letting $b(t)=\int_0^t\quot{\GLin_m\mathbb{R}}^*\beta(\dot{\gamma}(\tau))\mathrm{d}\tau$, the curve $\hat{\gamma}:t\mapsto\gamma(t)\exp(-b(t)\mathds{1})$ satisfies
\begin{align*}\omega_{+\beta}(\dot{\hat{\gamma}}(t))
& =\Ad_{\exp(b(t)\mathds{1})}\omega_{+\beta}(\dot{\gamma}(t))-b'(t)\mathds{1} \\
& =\Ad_{\exp(b(t)\mathds{1})}\left(\omega(\dot{\gamma}(t)) +\quot{\GLin_m\mathbb{R}}^*\beta(\dot{\gamma}(t))\mathds{1}\right)-\quot{\GLin_m\mathbb{R}}^*\beta(\dot{\gamma}(t))\mathds{1} \\
& =\mathrm{e}^{b(t)}v,\end{align*}
so the geodesic $t\mapsto\exp(t\omega_{+\beta}^{-1}(v))\mathscr{e}$ for the deformed geometry $(\mathscr{G},\omega_{+\beta})$ is just a reparametrization of $\hat{\gamma}$.

Exponential curves for $\omega_{+\beta}$ other than geodesics are not necessarily just reparametrizations of exponential curves for $\omega$ in quite the same way, as curves that are not lines through the origin usually are changed by rescaling the geometry. However, whenever we have an exponential curve $\gamma$ for $\omega_{+\beta}$, we can still relate the derivatives of the $\mathfrak{m}$-horizontal lift of $\quot{\GLin_m\mathbb{R}}(\gamma)$ with respect to $\omega$ to the derivatives of
\[\quot{\GLin_m\mathbb{R}}\big(\exp(t(v,R))\big)=\int_0^t\Ad_{\exp(\tau R)}(v)\mathrm{d}\tau.\]

\begin{lemma}\label{modifiedexp} Let $(\mathscr{G},\omega)$ be a Cartan geometry of type $(\Aff(m),\GLin_m\mathbb{R})$ over $M$, $\beta$ be a $1$-form on $M$, and $(v,R)\in\mathfrak{aff}(m)$. For the exponential\linebreak curve $\gamma:t\mapsto\exp(t\omega_{+\beta}^{-1}(v,R))\mathscr{e}$ for $\omega_{+\beta}$ and $b(t):=\int_0^t\quot{\GLin_m\mathbb{R}}^*\beta(\dot{\gamma}(\tau))\mathrm{d}\tau$, the curve $\gamma_\text{hor}:t\mapsto\gamma(t)\exp(b(t)\mathds{1}-tR)$ satisfies
\[\omega(\dot{\gamma}_\text{hor}(t))=\mathrm{e}^{-b(t)}\Ad_{\exp(tR)}(v)\in\mathfrak{m}=\mathbb{R}^m.\]
Therefore, for $s(t):=\int_0^t\mathrm{e}^{-b(r)}\mathrm{d}r$, $\omega(\tfrac{\mathrm{d}}{\mathrm{d}t}(\gamma_\text{hor}\circ s^{-1}(t)))=\Ad_{\exp(s^{-1}(t)R)}(v)$.\end{lemma}
\begin{proof}Since $\omega(\dot{\gamma})=(v,R)-\quot{\GLin_m\mathbb{R}}^*\beta(\dot{\gamma})\mathds{1}$,
\begin{align*}\omega(\dot{\gamma}_\text{hor}(t))
& =\Ad_{\exp(tR-b(t)\mathds{1})}\omega(\dot{\gamma}(t))+(b'(t)\mathds{1}-R)\\
& =\Ad_{\exp(tR-b(t)\mathds{1})}\left((v,R)-\quot{\GLin_m\mathbb{R}}^*\beta(\dot{\gamma})\mathds{1}\right)+\quot{\GLin_m\mathbb{R}}^*\beta(\dot{\gamma})\mathds{1}-R \\
& =\mathrm{e}^{-b(t)}\Ad_{\exp(tR)}(v).\end{align*}
Moreover, $s'(t)=\mathrm{e}^{-b(t)}$ and $\tfrac{\mathrm{d}}{\mathrm{d}t}(\gamma_\text{hor}\circ s^{-1}(t))=(s'(s^{-1}(t)))^{-1}\dot{\gamma}_\text{hor}(s^{-1}(t))$, so $\omega(\tfrac{\mathrm{d}}{\mathrm{d}t}(\gamma_\text{hor}\circ s^{-1}(t)))=\Ad_{\exp(s^{-1}(t)R)}(v)$.\mbox{\qedhere}\end{proof}

Notably, if $\beta$ is chosen so that geodesics are always forced to leave the support of $\beta$ and never return, then geodesic completeness of $(\mathscr{G},\omega)$ implies geodesic completeness of $(\mathscr{G},\omega_{+\beta})$. Likewise, if exponential curves for $\omega_{+\beta}$ must always eventually leave the support of $\beta$ and never return, then exp-completeness of $(\mathscr{G},\omega)$ implies exp-completeness of $(\mathscr{G},\omega_{+\beta})$.

\begin{proposition}\label{loopexamples1} If $\beta$ is a $1$-form on $\mathbb{R}^m$ with compact support, then the affine geometry $(\Aff(m),(\MC{\Aff(m)})_{+\beta})$ on $\mathbb{R}^m$ is exp-complete.\end{proposition}
\begin{proof}Suppose, toward contradiction, that
\[\gamma:[0,\tau)\to\Aff(m),~ t\mapsto\exp\left(t(\MC{\Aff(m)})_{+\beta}^{-1}(v,R)\right)\!(x_0,A_0)\]
is an exponential curve for $(\Aff(m),(\MC{\Aff(m)})_{+\beta})$, starting at the element $\gamma(0)=(x_0,A_0)$, that does not continuously extend to $\tau$. Using the notation from Lemma \ref{modifiedexp}, if $s(\tau)\in[0,+\infty)$ were well-defined, then the $\mathrm{g}_{\MC{\Aff(m)}}$-length of $\gamma_\text{hor}$ would be at most $s(\tau)\max_{t\in[0,\tau]}\|\Ad_{\exp(tR)}v\|_\mathrm{g}$, so $\gamma_\text{hor}$ would continuously extend to $\tau$ by metric completeness of the Klein geometry $(\Aff(m),\MC{\Aff(m)})$. In particular, $b|_{[0,\tau)}$ would be bounded, so $\gamma(t)=\gamma_\text{hor}(t)\exp(tR-b(t)\mathds{1})$ would continuously extend to $\tau$ as well, which would be a contradiction.

For $\gamma$ to be inextensible, we would therefore need $s([0,\tau))=[0,+\infty)$, but this also leads to a contradiction. Indeed, because $\gamma$ would not be inextensible if $v=0$, there would need to exist $t_0\in[0,\tau)$, $c>0$, and a linear functional $\ell\in(\mathbb{R}^m)^\vee$ such that $\ell(A_0\Ad_{\exp(tR)}v)\geq c$ for all $t\in[t_0,\tau)$. However,
\[\quot{\GLin_m\mathbb{R}}(\gamma(t))=\quot{\GLin_m\mathbb{R}}(\gamma_\text{hor}(t))=x_0+A_0\int_0^t\mathrm{e}^{-b(r)}\Ad_{\exp(rR)}v\,\mathrm{d}r,\]
so for $t\in[t_0,\tau)$,
\begin{align*}\ell\big(\quot{\GLin_m\mathbb{R}}(\gamma(t))\big)
& =\ell\big(\quot{\GLin_m\mathbb{R}}(\gamma(t_0))\big)+\ell\big(\quot{\GLin_m\mathbb{R}}(\gamma(t))-\quot{\GLin_m\mathbb{R}}(\gamma(t_0))\big) \\
& \geq\ell\big(\quot{\GLin_m\mathbb{R}}(\gamma(t_0))\big)+c\big(s(t)-s(t_0)\big),\end{align*}
which escapes to infinity as $t\rightarrow\tau$. For all $t$ sufficiently close to $\tau$, this would imply that $\quot{\GLin_m\mathbb{R}}(\gamma(t))$ lies outside of $\mathrm{supp}(\beta)$, hence that $b'(t)=0$, which would contradict that $s([0,\tau))=[0,+\infty)$ because $b$ would need to stay bounded.\mbox{\qedhere}\end{proof}

\subsection{Detecting dev-incompleteness via holonomy}
One especially convenient feature of rescaling deformations is that the linear holonomy of $(\mathscr{G},\omega)$ can often detect the failure of dev-completeness for $(\mathscr{G},\omega_{+\beta})$.

\begin{proposition}\label{holseesdc} Suppose $(\mathscr{G},\omega)$ is an affine geometry over $M$ and $\gamma:\mathbb{R}/\mathbb{Z}\to M$ is a smooth loop. Let $\gamma_\text{hor}:\mathbb{R}\to\mathscr{G}$ be an $\mathfrak{m}$-horizontal lift of $\gamma$ with respect to $\omega$, so that $\omega(\dot{\gamma}_\text{hor}(t))\in\mathfrak{m}$ for all $t\in\mathbb{R}$ and $\gamma_\text{hor}(t+1)=\gamma_\text{hor}(t)h_\gamma$ for some $h_\gamma\in\GLin_m\mathbb{R}$. For a 1-form $\beta$ on $M$, if every eigenvalue of $h_\gamma$ has absolute value greater than $\mathrm{e}^{\int_\gamma\beta}$, then the affine geometry $(\mathscr{G},\omega_{+\beta})$ is not dev-complete.\end{proposition}
\begin{proof}Let $b(t):=\int_0^t\quot{\GLin_m\mathbb{R}}^*\beta(\dot{\gamma}_\text{hor}(\tau))\mathrm{d}\tau$, and define a curve $\hat{\gamma}:\mathbb{R}\to\mathscr{G}$ by $\hat{\gamma}(t):=\gamma_\text{hor}(t)\exp(-b(t)\mathds{1})$,
so that
\begin{align*}\omega(\dot{\hat{\gamma}}(t))
& =\Ad_{\exp(b(t)\mathds{1})}\omega(\dot{\gamma}_\text{hor}(t))-b'(t)\mathds{1} \\
& =\mathrm{e}^{b(t)}\omega(\dot{\gamma}_\text{hor}(t))-\beta(\dot{\gamma}(t))\mathds{1}.\end{align*}
Then,
\[\omega_{+\beta}(\dot{\hat{\gamma}}(t))
=\omega(\dot{\hat{\gamma}}(t))+\beta(\dot{\gamma}(t))\mathds{1}
=\mathrm{e}^{b(t)}\omega(\dot{\gamma}_\text{hor}(t)),\]
so
\begin{align*}\omega_{+\beta}(\dot{\hat{\gamma}}(t+1))
& =\mathrm{e}^{b(t+1)}\omega(\dot{\gamma}_\text{hor}(t+1)) \\
& =\mathrm{e}^{\int_\gamma\beta}\mathrm{e}^{b(t)}\omega(\Rt{h_\gamma*}\dot{\gamma}_\text{hor}(t)) \\
& =\mathrm{e}^{\int_\gamma\beta}\Ad_{h_\gamma^{-1}}\left(\omega_{+\beta}(\dot{\hat{\gamma}}(t))\right).\end{align*}

Since every eigenvalue of $h_\gamma$ has absolute value greater than $\mathrm{e}^{\int_\gamma\beta}$, the linear transformation $\mathrm{e}^{\int_\gamma\beta}\Ad_{h_\gamma^{-1}}$ is contracting on $\mathfrak{m}$. In particular, if $\lambda$ is the largest absolute value of an eigenvalue of $\mathrm{e}^{\int_\gamma\beta}\Ad_{h_\gamma^{-1}}$ on $\mathfrak{m}=\mathbb{R}^m$, then for $r\in(\lambda,1)$ and $c\geq 1$ such that $\|(\mathrm{e}^{\int_\gamma\beta}\Ad_{h_\gamma^{-1}})^i(v)\|_\mathrm{g}\leq cr^i\|v\|_\mathrm{g}$,
\begin{align*}\int_0^{+\infty}\|\omega_{+\beta}(\dot{\hat{\gamma}}(t))\|_\mathrm{g}\mathrm{d}t
& \leq\sum_{i=0}^{+\infty}cr^i\int_0^1\|\omega_{+\beta}(\dot{\hat{\gamma}}(t))\|_\mathrm{g}\mathrm{d}t \\
& =\frac{c}{1-r}\int_0^1\|\omega_{+\beta}(\dot{\hat{\gamma}}(t))\|_\mathrm{g}\mathrm{d}t,\end{align*}
so the development $\hat{\gamma}_{\Aff(m)}:[0,+\infty)\to\mathbb{R}^m<\Aff(m)$ of $\hat{\gamma}$ with respect to $\omega_{+\beta}$ has finite length. By $\MC{\Aff(m)}$-constant metric completeness of the Klein geometry on $\Aff(m)$, $\hat{\gamma}_{\Aff(m)}(t)$ converges as $t\rightarrow +\infty$, so $\hat{\gamma}_{\Aff(m)}\circ(t\mapsto\tfrac{t}{1-t})$ is a well-defined path in $\Path_{(0,\mathds{1})}(\Aff(m))$ with no antidevelopment starting at $\hat{\gamma}(0)$ in $(\mathscr{G},\omega_{+\beta})$.\mbox{\qedhere}\end{proof}

The linear holonomy cannot always detect dev-incompleteness in this way. Indeed, the $\mathfrak{aff}(1)$-valued deformation $\phi$ we saw in Section \ref{halfinated} is a rescaling deformation for the 1-form $\beta=-\mathrm{d}x$, and in that case, the integral of $\beta$ around any loop is $0$ because $-\mathrm{d}x$ is exact. We can, however, show that rescaling deformations for $\beta=\mathrm{d}f$ cannot break dev-completeness when $f$ is a bounded function, which will lead directly to the second part of the main result for this section.

\begin{lemma}\label{exactbounded} If $(\mathscr{G},\omega)$ is a dev-complete affine geometry over $M$ and $f:M\to\mathbb{R}$ is a bounded smooth function on $M$, then $(\mathscr{G},\omega_{+\mathrm{d}f})$ is also dev-complete.\end{lemma}
\begin{proof}Consider the $\GLin_m\mathbb{R}$-equivariant diffeomorphism
\[\sigma:\mathscr{G}\to\mathscr{G},~\mathscr{g}\mapsto\mathscr{g}\exp(-f(\quot{\GLin_m\mathbb{R}}(\mathscr{g}))\mathds{1}),\]
Then, $\sigma^*\omega_{+\mathrm{d}f}=\Ad_{\exp(f\mathds{1})}\omega$, so for $\xi\in T\mathscr{G}$ and $\mathrm{g}$ an inner product on $\mathfrak{aff}(m)$ with respect to which $\mathfrak{m}=\mathbb{R}^m$ and $\mathfrak{gl}_m\mathbb{R}$ are perpendicular, $\|\xi\|_{\mathrm{g}_{\sigma^*\omega_{+\mathrm{d}f}}}\geq\mathrm{e}^{-2\sup|f|}\|\xi\|_{\mathrm{g}_\omega}$. Since $\mathrm{g}_\omega$ is a complete metric, this tells us that $\sigma^*\mathrm{g}_{\omega_{+\mathrm{d}f}}=\mathrm{g}_{\sigma^*\omega_{+\mathrm{d}f}}$ must also be complete, so $\mathrm{g}_{\omega_{+\mathrm{d}f}}$ is complete.\mbox{\qedhere}\end{proof}

\begin{corollary}\label{loopexamples2} If $\beta$ is a $1$-form on $\mathbb{R}^m$ with compact support, then the affine geometry $(\Aff(m),(\MC{\Aff(m)})_{+\beta})$ on $\mathbb{R}^m$ is dev-complete if and only if $\mathrm{d}\beta=0$.\end{corollary}
\begin{proof}Suppose $\beta$ is a $1$-form on $\mathbb{R}^m$ with compact support. If $\mathrm{d}\beta=0$, then $\beta$ is exact and there is some bounded function $f:\mathbb{R}^m\to\mathbb{R}$ such that $\beta=\mathrm{d}f$, so $(\Aff(m),(\MC{\Aff(m)})_{+\mathrm{d}f})$ is dev-complete by Lemma \ref{exactbounded}. If $\mathrm{d}\beta\neq 0$, then there is some loop $\gamma:[0,1]\to\mathbb{R}^m$ with $\int_\gamma\beta<0$, so Proposition \ref{holseesdc} tells us, since $h_\gamma=\mathds{1}$ for all loops in $(\Aff(m),\MC{\Aff(m)})$, that $(\Aff(m),(\MC{\Aff(m)})_{+\beta})$ is not dev-complete.\mbox{\qedhere}\end{proof}

Combined with Proposition \ref{loopexamples1}, this gives an infinite-dimensional family of affine examples that are exp-complete but not dev-complete.

\bibliographystyle{plain}
\bibliography{completeness-refs}

\end{document}